\documentclass[12pt,draftcls,onecolumn]{subIEEEtran}

\newtheorem{theorem}{Theorem}

\newtheorem{proposition}[theorem]{Proposition}

\newcommand{\enma}[1]   {\ensuremath{#1}}

\newcommand{\non}{\nonumber}

\newcommand{\beq}{\begin{equation}}
\newcommand{\eeq}{\end{equation}}
\newcommand{\bseq}{\begin{subequations}}
\newcommand{\eseq}{\end{subequations}}
\newcommand{\beqn}{\begin{eqnarray}}
\newcommand{\eeqn}{\end{eqnarray}}
\newcommand{\ba}{\begin{array}}
\newcommand{\ea}{\end{array}}
\newcommand{\bct}{\begin{center}}
\newcommand{\ect}{\end{center}}
\newcommand{\btmz}{\begin{itemize}}
\newcommand{\etmz}{\end{itemize}}
\newcommand{\benum}{\begin{enumerate}}
\newcommand{\eenum}{\end{enumerate}}

\newcommand{\diag}      {\enma{\mathrm{diag}}}

\newcommand{\trace}     {\enma{\mathrm{trace}}}

\newcommand{\matbegin}{
        \left[
}
\newcommand{\matend}{
        \right]
}

\newcommand{\tbo}[2]{
  \matbegin \begin{array}{c}
       #1 \\ #2
       \end{array} \matend }

\newcommand{\obt}[2]{
  \matbegin \begin{array}{cc}
       #1 & #2
       \end{array} \matend }

\newcommand{\tbt}[4]{
  \matbegin \begin{array}{cc}
       #1 & #2 \\ #3 & #4
       \end{array} \matend }

\newcommand{\thbth}[9]{
 \matbegin \begin{array}{ccc}
                #1 & #2 & #3 \\
                #4 & #5 & #6 \\
                #7 & #8 & #9
                \end{array}\matend}

\newcommand{\be}{\begin{equation}}
\newcommand{\ee}{\end{equation}}

\newcommand{\cplxs}{ C\kern -.35em \rule{0.03 em}{.7 ex}~   }

\def\complex{\hbox{C\kern -.45em \rule{0.03 em}{1.5 ex}}~}

\newcommand{\bi}{\begin{itemize}}
\newcommand{\ei}{\end{itemize}}

\usepackage{amssymb,latexsym,color,amsmath,pifont,epsfig,graphicx,dsfont}
\usepackage{setspace,cite}
\usepackage{subfig}  
\usepackage{algorithm}
\usepackage{algorithmic}
\usepackage{tikz}
\usepackage{pgfplots}
\usetikzlibrary{plotmarks}

\newcommand{\eN}{{\cal N}}

\newcommand{\ds}{\displaystyle}

\newcommand{\tc}{\textcolor}
\newcommand{\dso}{\mathds{1}}

\definecolor{darkgreen}{rgb}{0,0.4,0}
\newcommand{\DefinedAs}[0]{\mathrel{\mathop:}=}

\IEEEoverridecommandlockouts                              
\begin{document}

\title{\LARGE \bf Algorithms for Leader Selection in \\[-0.25cm] Stochastically Forced Consensus Networks}

\author{Fu Lin, Makan Fardad, and Mihailo R. Jovanovi\'c
\thanks{Financial support from the National Science Foundation under CAREER Award CMMI-06-44793 and under awards CMMI-09-27720 and CMMI-0927509 is gratefully acknowledged.}
\thanks{F.\ Lin and M.\ R.\ Jovanovi\'c are with the Department of Electrical and Computer Engineering, University of Minnesota, Minneapolis, MN  55455. M.\ Fardad is with the Department of Electrical Engineering and Computer Science, Syracuse University, NY 13244. E-mails: fu@umn.edu, makan@syr.edu, mihailo@umn.edu.}}

\maketitle

    \vspace{-9ex}

    \begin{abstract}
    \vspace{-2ex}
We are interested in assigning a pre-specified number of nodes as leaders in order to minimize the mean-square deviation from consensus in stochastically forced networks. This problem arises in several applications including control of vehicular formations and localization in sensor networks. For networks with leaders subject to noise, we show that the Boolean constraints (a node is either a leader or it is not) are the only source of nonconvexity. By relaxing these constraints to their convex hull we obtain a lower bound on the global optimal value. We also use a simple but efficient greedy algorithm to identify leaders and to compute an upper bound. For networks with leaders that perfectly follow their desired trajectories, we identify an additional source of nonconvexity in the form of a rank constraint. Removal of the rank constraint and relaxation of the Boolean constraints yields a semidefinite program for which we develop a customized algorithm well-suited for large networks. Several examples ranging from regular lattices to random graphs are provided to illustrate \mbox{the effectiveness of the developed algorithms.}
    \end{abstract}

    \vspace{-2ex}

    \begin{keywords}
    \vspace{-2ex}
Alternating direction method of multipliers, consensus networks, convex optimization, convex relaxations, greedy algorithm, leader selection, performance bounds, semidefinite programming, sensor selection, variance amplification.
    \end{keywords}


\vspace*{-4ex}
\section{Introduction}

Reaching consensus in a decentralized fashion is an important problem in network science~\cite{mesege10}. This problem is often encountered in social networks where a group of individuals is trying to agree on a certain issue~\cite{deg74,goljac10}. A related load balancing problem has been studied extensively in computer science with the objective of distributing evenly computational load over a network of processors~\cite{cyb89,boi90}. Recently, consensus problem has received considerable attention in the context of distributed control~\cite{jadlinmor03,carfagspezam07}. For example, in cooperative control of vehicular formations, it is desired to use local interactions between vehicles in order to reach agreement on quantities such as heading angle, velocity, and inter-vehicular spacing. Since vehicles have to maintain agreement in the presence of uncertainty, it is important to study robustness of consensus. Several authors have recently used the steady-state variance of the deviation from consensus to characterize fundamental performance limitations in stochastically forced networks~\cite{barhes07,barhes08,xiaboykim07,youscaleo10,zelmes11,bamjovmitpat12,linfarjovTAC12platoons}.

In this paper, we consider undirected consensus networks with two groups of nodes. Ordinary nodes, the so-called {\em followers\/}, form their action using relative information exchange with their neighbors; special nodes, the so-called {\em leaders\/}, also have access to their own states. This setting may arise in the control of vehicular formations and in distributed localization in sensor networks. In vehicular formations, all vehicles are equipped with ranging devices (that provide information about relative distances with respect to their neighbors), and the leaders additionally have GPS devices (that provide information with respect to a global frame of reference).

We are interested in assigning a pre-specified number of nodes as leaders in order to minimize the mean-square deviation from consensus. For undirected networks in which all nodes are subject to stochastic disturbances, we show that the Boolean constraints (a node is either a leader or it is not) are the only source of nonconvexity. The combinatorial nature of these constraints makes determination of the global minimum challenging. Instead, we focus on computing lower and upper bounds on the global optimal value. Convex relaxation of Boolean constraints is used to obtain a lower bound, and a greedy algorithm is used to obtain an upper bound and to identify leaders. We show that the convex relaxation can be formulated as a semidefinite program~(SDP) which can be solved efficiently for small networks. We also develop an efficient customized interior point method that is well-suited for large-scale problems. Furthermore, we improve performance of one-leader-at-a-time (greedy) approach using a procedure that checks for possible swaps between leaders and followers. In both steps, algorithmic complexity is significantly reduced by exploiting the structure of low-rank modifications to Laplacian matrices. The computational efficiency of our algorithms makes them well-suited for establishing achievable performance \mbox{bounds for leader selection problem in large stochastically forced networks.}

Following~\cite{patbam10,clapoo11,clabuspoo12,kawege12}, we also examine consensus networks in which leaders follow desired trajectories at all times. For consensus networks with at least one leader, adding leaders always improves performance~\cite{patbam10}. In view of this, a greedy algorithm that selects one leader at a time by assigning the node that leads to the largest performance improvement as a leader was proposed in~\cite{patbam10}. Furthermore, it was proved in~\cite{clabuspoo12} that the mean-square deviation from consensus is a supermodular function of the set of noise-free leaders. Thus, the supermodular optimization framework in conjunction with the greedy algorithm can be used to provide selection of leaders that is within a provable \mbox{bound from globally optimal solution~\cite{clabuspoo12}.}

In contrast to~\cite{patbam10,clapoo11,clabuspoo12,kawege12}, we use convex optimization to quantify performance bounds for the noise-free leader selection problem. While we show that the leader selection is additionally complicated by the presence of a nonconvex rank constraint, we obtain an SDP relaxation by dropping the rank constraint and by relaxing the aforementioned Boolean constraints. Furthermore, we exploit the separable structure of the resulting constraint set and develop an efficient algorithm based on the alternating direction method of multipliers~(ADMM). As in the noise-corrupted problem, we use a greedy approach followed by a swap procedure to compute an upper bound and to select leaders. In both steps, we exploit the properties of low-rank modifications to Laplacian matrices to reduce computational complexity.

Several recent efforts have focused on characterizing graph-theoretic conditions for controllability of networks in which a pre-specified number of leaders act as control inputs~\cite{tan04,liuchuwanxie08,rahjimesege09,jiawanlinwan09,clabuspooCDC12,yazabbege12}. In contrast, our objective is to identify leaders that are most effective in minimizing the deviation from consensus in the presence of disturbances. Several alternative performance indices for the selection of leaders have been also recently examined in~\cite{claalobuspoo12,clabuspooCDC12,clabuspoo13}. Other related work on augmenting topologies of networks to improve their algebraic connectivity includes~\cite{ghoboy06,zelschall12}.

We finally comment on the necessity of considering two different problem formulations. The noise-free leader selection problem, aimed at identifying influential nodes in undirected networks, was originally formulated in~\cite{patbam10} and consequently studied in~\cite{clapoo11,clabuspoo12,kawege12}. To the best of our knowledge, the noise-corrupted leader selection problem first appeared in a preliminary version of this work~\cite{linfarjovCDC11}. The noise-corrupted formulation is introduced for two reasons: First, it is well-suited for applications where a certain number of nodes are to be equipped with additional capabilities (e.g., the GPS devices) in order to improve the network's performance; for example, this setup may be encountered in vehicular formation and sensor localization problems. And, second, in contrast to the noise-free formulation, the Boolean constraints are the only source of nonconvexity in the noise-corrupted problem; consequently, a convex relaxation in this case is readily obtained by enlarging Boolean constraints to their convex hull. Even though these formulations have close connections in a certain limit, the differences between them are significant enough to warrant separate treatments. As we show in Sections~\ref{sec.noise_corrupted} and~\ref{sec.noisefree}, the structure of the corresponding optimization problems necessitates separate convex relaxations and the development of different customized optimization algorithms. Noise-free and noise-corrupted setups are thus of independent interest \mbox{from the application, problem formulation, and algorithmic points of view.}

The paper is organized as follows. In Section~\ref{sec.problem}, we formulate the problem and establish connections between the leader selection and the sensor selection problems. In Section~\ref{sec.noise_corrupted}, we develop efficient algorithms to compute lower and upper bounds on the global optimal value for the noise-corrupted leader selection problem. In Section~\ref{sec.noisefree}, we provide an SDP relaxation of the noise-free formulation and employ the ADMM algorithm to deal with large-scale problems. We conclude the paper with a summary of our contributions in Section~\ref{sec.conclude}.

    \vspace*{-2ex}
\section{Problem formulation}
    \label{sec.problem}

In this section, we formulate the noise-corrupted and noise-free leader selection problems in consensus networks and make connections to sensor selection in distributed localization problems. Furthermore, we establish an equivalence between two problem formulations when all leaders use arbitrarily large feedback gains on their own states.

	\vspace*{-2ex}
\subsection{Leader selection problem in consensus networks}

We consider networks in which each node updates a scalar state $\psi_i$,
    \[
        \dot{\psi}_i
        \; = \;
        u_i
        \; + \;
        w_i,
        ~~
        i \,=\, 1,\ldots,n
    \]
where $u_i$ is the control input and $w_i$ is the white stochastic disturbance with zero-mean and unit-variance. A node is a {\em follower} if it uses {\em only\/} relative information exchange with its neighbors to form its control action,
    \beq
    \non
        u_i
        \,=\,
        - \sum_{j\,\in\,\eN_i}
        ( \psi_i \,-\, \psi_j ).
    \eeq
A node is a {\em leader\/} if, in addition to relative information exchange with its neighbors, it also has access to its own state
    \[
        u_i
        \,=\,
        - \sum_{j\,\in\,\eN_i}
        ( \psi_i \,-\, \psi_j )
        \,-\,
        \kappa_i \, \psi_i.
    \]
Here, $\kappa_i$ is a positive number and ${\cal N}_i$ is the set of all nodes that node $i$ communicates with.


The control objective is to strategically deploy leaders in order to reduce the variance amplification in stochastically forced consensus networks. The communication network is modeled by a connected, undirected graph; thus, the graph Laplacian $L$ is a symmetric positive semidefinite matrix with a single eigenvalue at zero and the corresponding eigenvector $\dso$ of all ones~\cite{mesege10}. A state-space representation of the leader-follower consensus network is therefore given by
    \beq
    \label{eq.state}
        \dot{\psi}
        \,=\,
        - \,
        (L \,+\, D_\kappa D_x )
        \, \psi
        \,+\,
        w
    \eeq
where
	$
	{\cal E}
        	\left(
        	w(t) \, w^T(\tau)
        	\right)
	=
	I \, \delta (t - \tau),
	$
${\cal E}(\cdot)$ is the expectation operator, and
    \beq
    \non
        D_\kappa \, \DefinedAs \, \diag \left( \kappa \right),
        ~~~
        D_x  \, \DefinedAs \, \diag \left( x \right)
    \eeq
are diagonal matrices formed from the vectors $\kappa = [\, \kappa_1 ~\cdots~ \kappa_n \,]^T$ and $x = [\, x_1 ~ \cdots ~ x_n \,]^T$. Here, $x$ is a Boolean-valued vector with its $i$th entry $x_i \in \{0, \, 1\}$, indicating that node $i$ is a leader if $x_i = 1$ and that node $i$ is a follower if $x_i = 0$. In connected networks with at least one leader, $L + D_\kappa D_x$ is a positive definite matrix~\cite{rahjimesege09}. The steady-state covariance matrix of $\psi$
    \[
        \Sigma
        \, \DefinedAs \,
        \ds{\lim_{t \, \to \, \infty}}
        {\cal E}
        \left(
        \psi(t) \, \psi^T(t)
        \right)
    \]
can thus be determined from the Lyapunov equation
    \beq
        (L + D_\kappa D_x)
        \,
        \Sigma
        \, + \,
        \Sigma
        \,
        (L + D_\kappa D_x)
        \; = \;
        I
        \non
    \eeq
whose unique solution is given by
    \[
        \Sigma \,=\, \dfrac{1}{2} \, (L + D_\kappa D_x)^{-1}.
    \]
Following~\cite{xiaboykim07,bamjovmitpat12}, we use the total steady-state variance
    \beq
    \label{eq.variance}
        \trace
        \left(
        \Sigma
        \right)
        \,=\,
        \dfrac{1}{2} \, \trace \left( (L + D_\kappa D_x)^{-1} \right)
    \eeq
to quantify performance of stochastically forced consensus networks.

We are interested in identifying $N_l$ leaders that are most effective in reducing the steady-state variance~(\ref{eq.variance}). For an {\em a priori\/} specified number of leaders $N_l < n$, the leader selection problem can thus be formulated as
    \beq
    \label{LS1}
    \tag{LS1}
    \ba{lrcl}
    \underset{x}{\mbox{minimize}}
    &
    J(x)
    & = &
    \trace
    \left(
    ( L \,+\, D_\kappa D_x )^{-1}
    \right)
    \\[0.1cm]
    \text{subject to}
    &
    x_{i}
    & \in &
    \{0,1\} ,
    ~~~~~
    i \,=\, 1,\ldots,n
    \\
    &
    \dso^T x
    & = & N_l.
    \ea
    \eeq
In~\eqref{LS1}, the number of leaders $N_l$ as well as the matrices $L$ and $D_\kappa$ are the problem data, and the vector $x$ is the optimization variable. As we show in Section~\ref{sec.noise_corrupted}, for a positive definite matrix $L + D_\kappa D_x$, the objective function $J$ in~\eqref{LS1} is a convex function of $x$. The challenging aspect of~\eqref{LS1} comes from the nonconvex Boolean constraints $x_{i} \in \{0,1\}$; in general, finding the solution to~\eqref{LS1} requires an intractable combinatorial search.

Since the leaders are subject to stochastic disturbances, we refer to~\eqref{LS1} as the {\em noise-corrupted\/} leader selection problem. We also consider the selection of {\em noise-free\/} leaders which follow their desired trajectories at all times~\cite{patbam10}. Equivalently, in coordinates that determine deviation from the desired trajectory, the state of every leader is identically equal to zero, and the network dynamics are thereby governed by the dynamics of the followers
    \[
        \dot{\psi}_f
        \; = \;
        - \, L_f \, \psi_f
        \; + \;
        w_f.
    \]
Here, $L_f$ is obtained from $L$ by eliminating all rows and columns associated with the leaders. Thus, the problem of selecting leaders that minimize the steady-state variance of $\psi_f$ amounts to
    \[
        \ba{lrcl}
        \underset{x}{\mbox{minimize}}
        &
        J_f(x)
        & = &
        \trace \, (L_f^{-1})
        \\[0.15cm]
        \text{subject to}
        &
        x_i
        & \in &
        \{0,1\},
        ~~~~~
        i \; = \; 1,\ldots,n
        \\
        &
        \dso^T x & = & N_l.
        \ea
        \tag{LS2}
        \label{LS2}
    \]
As in~\eqref{LS1}, the Boolean constraints $x_{i} \in \{0,1\}$ are nonconvex. Furthermore, as we demonstrate in Section~\ref{sec.convex_noisefree}, the objective function $J_f$ in~\eqref{LS2} is a nonconvex function of $x$.

We note that the noise-free leader selection problem~\eqref{LS2} cannot be uncovered from the noise-corrupted leader selection problem~\eqref{LS1} by setting the variance of disturbances (that act on noise-corrupted leaders) to zero. Even when leaders are not directly subject to disturbances, their interactions with followers would prevent them from perfectly following their desired trajectories. In what follows, we establish the equivalence between the noise-corrupted and noise-free leader selection problems~\eqref{LS1} and~\eqref{LS2} in the situations when all noise-corrupted leaders use arbitrarily large feedback gains on their own states. Specifically, for white in time stochastic disturbance $w$ with unit variance, the variance of noise-corrupted leaders in~\eqref{eq.state} decreases to zero as feedback gains on their states increase to infinity; see Appendix~\ref{app.nf-vs-nc}.

	\vspace*{-2ex}
\subsection{Connections to the sensor selection problem}
    \label{sec.sensor-selection}

The problem of estimating a vector $\psi \in \mathbb{R}^n$ from $m$ relative measurements that are corrupted by additive white noise
    \[
        y_{ij} \,=\, \psi_i \,-\, \psi_j \,+\, w_{ij}
    \]
arises in distributed localization in sensor networks. We consider the simplest scenario in which all $\psi_i$'s are scalar-valued, with $\psi_i$ denoting the position of sensor $i$; see~\cite{barhes07,barhes08} for vector-valued localization problems. Let ${\cal I}_r$ denote the index set of the $m$ pairs of distinct nodes between which the relative measurements are taken and let $e_{ij}$ belong to $\mathbb{R}^n$ with $1$ and $-1$ at its $i$th and $j$th elements, respectively, and zero everywhere else. Then,
    \[
        y_{ij}
        \,=\,
        e_{ij}^T \,\psi
        \,+\,
        w_{ij},
        ~~~
        (i,j) \, \in \, {\cal I}_r
    \]
or, equivalently in the matrix form,
    \beq
    \label{eq.relative}
        y_r \, = \, E_r^T \psi \, + \, w_r
    \eeq
where $y_r$ is the vector of relative measurements and $E_r \in \mathbb{R}^{n \times m}$ is the matrix whose columns are determined by $e_{ij}$ for $(i,j) \in {\cal I}_r$. Since $\psi + a \dso$ for any scalar $a$ results in the same $y_r$, use of relative measurements provides estimate of the position vector $\psi$ only up to an additive constant. This can be also verified by noting that $E_r^T \dso = 0$.

Suppose that $N_l$ sensors can be equipped with GPS devices that allow them to measure their absolute positions
    \[
        y_a
        \,=\,
        E_a^T \psi  \,+\, E_a^T w_a
    \]
where $E_a \in \mathbb{R}^{n \times N_l}$ is the matrix whose columns are determined by $e_i$, the $i$th unit vector in $\mathbb{R}^n$, for $i \in {\cal I}_a$, the index set of absolute measurements. Then the vector of all measurements is given by
    \be
    \label{eq.y}
        \tbo{y_r}{y_a}
        \,=\,
        \tbo{E^T_r}{E^T_a} \psi
        \,+\,
        \tbt{I}{0}{0}{E_a^T}
        \tbo{w_r}{w_a}
    \ee
where $w_r$ and $w_a$ are zero-mean white stochastic disturbances with
    \[
        {\cal E} ( w_r w_r^T ) \,=\, W_r,
        ~~~
        {\cal E} ( w_a w_a^T ) \,=\, W_a,
        ~~~
        {\cal E} ( w_r w_a^T ) \,=\, 0.
    \]

In Appendix~\ref{app.sensor}, we show that the problem of choosing $N_l$ absolute position measurements among $n$ sensors to minimize the variance of the estimation error is equivalent to the noise-corrupted leader selection problem~\eqref{LS1}. Furthermore, when the positions of $N_l$ sensors are known {\em a priori\/} we show that the problem of assigning $N_l$ sensors to minimize the variance of the estimation error amounts to solving the noise-free leader selection problem~\eqref{LS2}.

	\vspace*{-2ex}
\section{Lower and upper bounds on global performance: Noise-corrupted leaders}
    \label{sec.noise_corrupted}

In this section, we show that the objective function $J$ in the noise-corrupted leader selection problem~\eqref{LS1} is convex. Convexity of $J$ is utilized to develop efficient algorithms for  computation of lower and upper bounds on the global optimal value~\eqref{LS1}. A lower bound results from convex relaxation of Boolean constraints in~\eqref{LS1} which yields an SDP that can be solved efficiently using a customized interior point method. On the other hand, an upper bound is obtained using a greedy algorithm that selects one leader at a time. Since greedy algorithm introduces low-rank modifications to Laplacian matrices, we exploit this feature in conjunction with the matrix inversion lemma to gain computational efficiency. Finally, we provide two examples to illustrate performance of the developed approach.

	\vspace*{-2ex}
\subsection{Convex relaxation to obtain a lower bound}
    \label{sec.convex}

Since the objective function $J$ in~\eqref{LS1} is the composition of a convex function $\trace \, (\bar{L}^{-1})$ of a positive definite matrix $\bar{L} \succ 0$ with an affine function $\bar{L} \DefinedAs L + D_\kappa D_x$, it follows that $J$ is a convex function of $x$. By enlarging the Boolean constraint set $x_{i} \in \{0,1\}$ to its convex hull $x_{i} \in [0,1]$, we obtain the following convex relaxation of~\eqref{LS1}
    \beq
    \tag{CR1}
    \label{CR1}
    \ba{ll}
    \underset{x}{\mbox{minimize}}
    &
    J(x)
    \,=\,
    \trace \, \big( ( L \,+\, D_\kappa D_x )^{-1} \big)
    \\[0.1cm]
    \text{subject to}
    &
    \dso^T x \,=\, N_l,
    ~~~
    0 \, \leq \, x_i \, \leq \, 1,
    ~~~~~
    i \,=\, 1,\ldots,n.
    \ea
    \eeq
Since we have enlarged the constraint set, the solution $x^*$ of the relaxed problem~\eqref{CR1} provides a lower bound on $J_{\rm opt}$. However, $x^*$ may not provide a selection of $N_l$ leaders, as it may not be Boolean-valued. If $x^*$ is Boolean-valued, then it is the global solution of~\eqref{LS1}.

Schur complement can be used to formulate the optimization problem~\eqref{CR1} as an SDP
    \beq
    \non
    \ba{ll}
    \underset{X, \; x}{\mbox{minimize}}
    &
    \trace \,( X ) \\[0.1cm]
    \text{subject to}
    &
    \tbt{X}{I}{I}{ L + D_\kappa D_x } \, \succeq \, 0
    \\[0.45cm]
    &
    \dso^T x \,=\, N_l,
    ~~~
    0 \, \leq \, x_i \, \leq \, 1,
    ~~~
    i \,=\, 1, \ldots, n.
    \ea
    \eeq
For small networks~(e.g., $n \leq 30$), this problem can be solved efficiently using standard SDP solvers. For large networks, we develop a customized interior point method in Appendix~\ref{app.IPM}.

	\vspace*{-2ex}
\subsection{Greedy algorithm to obtain an upper bound}
    \label{sec.greedy}

With the lower bound on the optimal value $J_{\rm opt}$ resulting from the convex relaxation~\eqref{CR1}, we next use a greedy algorithm to compute an upper bound on $J_{\rm opt}$. This algorithm selects one leader at a time by assigning the node that provides the largest performance improvement as the leader. Once this is done, an attempt to improve a selection of $N_l$ leaders is made by checking possible swaps between the leaders and the followers. In both steps, we show that substantial improvement in algorithmic complexity can be achieved by exploiting structure of the low-rank modifications to Laplacian matrices.

\subsubsection{One-leader-at-a-time algorithm}

As the name suggests, we select one leader at a time by assigning the node that results in the largest performance improvement as the leader.
For $i = 1, \ldots, n $, we compute
    \[
    J_1^i
    \;=\;
    \trace
    \left(
    (L \,+\, \kappa_i e_i e_i^T)^{-1}
    \right)
    \]
and assign the node, say $v_1$, that achieves the minimum value of $\{J_1^i\}$ as the first leader. If two or more nodes provide the optimal performance, we select one of these nodes as a leader. After choosing $s$ leaders, $v_1,\ldots,v_s$, we compute
    \[
    \ba{rrl}
    J_{s+1}^i
    & \! = \! &
    \trace
    \left(
    (L_s \,+\, \kappa_i e_i e_i^T)^{-1}
    \right)
    \\
    L_s
    & \! \DefinedAs \! &
    L
    \,+\,
    \kappa_{v_1}
    e_{v_1} e_{v_1}^T
    \,+\,
    \cdots
    \,+\,
    \kappa_{v_s} e_{v_s} e_{v_s}^T
    \ea
    \]
for $i \notin \{v_1,\ldots,v_s\}$, and select node $v_{s+1}$ that yields the minimum value of $\{J_{s+1}^i\}$ as the $(s+1)$th leader. This procedure is repeated until all $N_l$ leaders are selected.

Without exploiting structure, the above procedure requires $O(n^4 N_l)$ operations. On the other hand, the rank-$1$ update formula resulting from the matrix inversion lemma
    \beq
    \label{eq.rank1}
    (
    L_s
    \,+\,
    \kappa_i e_i e_i^T
    )^{-1}
    \; = \;
    L_s^{-1}
    \; - \;
    \dfrac
    {L_s^{-1} \, \kappa_i e_i e_i^T \, L_s^{-1}}
    {1 \,+\, \kappa_i e_i^T L_s^{-1} e_i}
    \eeq
yields
    \beq
    \non
    J_{s+1}^i
    \; = \;
    \trace
    \,(L_s^{-1})
    \; - \;
    \dfrac
    { \kappa_i \, \| (L_s^{-1})_i \|_2^2}
    {1 \, + \, \kappa_i (L_s^{-1})_{ii}}.
    \eeq
Here, $(L_s^{-1})_i$ is the $i$th column of $L_s^{-1}$ and $(L_s^{-1})_{ii}$ is the $ii$th entry of $L_s^{-1}$. To initiate the algorithm, we use the generalized \mbox{rank-$1$} update~\cite{mey73},
    \beq
    \non
    L_1^{-1}
    =
    L^\dagger
    \,-\,
    ( L^\dagger e_i ) \dso^T
    -\,
    \dso ( L^\dagger e_i )^T
    \,+\,
    ( (1/\kappa_i) \,+\, e_i^T  L^\dagger e_i )
    \dso \dso^T
    \eeq
which thereby yields,
    \beq
    \non
    J_1^i
    \,=\,
    \trace\,(L^\dagger)
    \,+\,
    n \,
    ((1/\kappa_i) \,+\, e_i^T L^\dagger e_i)
    \eeq
where $L^\dagger$ denotes the pseudo-inverse of $L$~(e.g., see~\cite{ghoboysab08})
    \[
    L^\dagger
    \, = \,
    (
    L
    \, + \,
    \dso \dso^T/n
    )^{-1}
    \, - \,
    \dso \dso^T/n.
    \]
Therefore, once $L_s^{-1}$ is determined, the inverse of the matrix on the left-hand-side of~(\ref{eq.rank1}) can be computed using $O(n^2)$ operations and $J_{s+1}^i$ can be evaluated using $O(n)$ operations. Overall, $N_l$ rank-1 updates, $nN_l/2$ objective function evaluations, and one full matrix inverse (for computing $L_s^{-1}$) require $O(n^2 N_l + n^3)$ operations as opposed to $O(n^4 N_l)$ operations without exploiting the low-rank structure. In large-scale networks, further computational advantage may be gained by exploiting structure of the underlying Laplacian matrices; e.g., see~\cite{spi10}.

\subsubsection{Swap algorithm}
    \label{sec.swap}

After leaders are selected using the one-leader-at-a-time algorithm, we swap one of the $N_l$ leaders with one of the $n - N_l$ followers, and check if such a swap leads to a decrease in $J$. If no decrease occurs for all $(n-N_l)N_l$ swaps, the algorithm terminates; if a decrease in $J$ occurs, we update the set of leaders and then check again the possible $(n-N_l)N_l$ swaps for the new leader selection. A similar swap procedure has been used as an effective means for improving performance of combinatorial algorithms encountered in graph partitioning~\cite{kerlin70}, sensor selection~\cite{josboy09}, and community detection problems~\cite{new06}.

Since a swap between a leader $i$ and a follower $j$ leads to a rank-$2$ modification~(\ref{eq.rank2}) to the matrix
   $
   \bar{L}
   \DefinedAs
   L
   +
   D_\kappa D_x,
   $
we can exploit this low-rank structure to gain computational efficiency. Using the matrix inversion lemma, we have
    \beq
    \ba{l}
    \left(
    \bar{L}
    \,-\,
    \kappa_i e_i e_i^T
    \,+\,
    \kappa_j e_j e_j^T
    \right)^{-1}
    = \;
    \bar{L}^{-1}
    \,-\,
    \bar{L}^{-1}
    \,
    \bar{E}_{ij}
    \,
    (
    I_2
    \,+\,
    E_{ij}^T
    \bar{L}^{-1}
    \bar{E}_{ij}
    )^{-1}
    \,
    E_{ij}^T
    \,
    \bar{L}^{-1}
    \ea
    \label{eq.rank2}
    \eeq
where
    $
    E_{ij}
    =
    [\,e_i ~~ e_j \,],
    ~
    \bar{E}_{ij}
    =
    [\,- \, \kappa_i e_i ~~ \kappa_j e_j \,],
    $
and $I_2$ is the $2\times 2$ identity matrix. Thus, the objective function after the swap between leader $i$ and follower $j$ is given by
    \beq
    \label{eq.Jij}
    J_{ij}
    \, = \,
    J
    \,-\,
    \trace
    \left(
    (
    I_2
    +
    E_{ij}^T
    \bar{L}^{-1}
    \bar{E}_{ij}
    )^{-1}
    E_{ij}^T
    \,
    \bar{L}^{-2}
    \bar{E}_{ij}
    \right).
    \eeq
Here, we do not need to form the full matrix $\bar{L}^{-2}$, since
    \[
    E_{ij}^T
    \,
    \bar{L}^{-2}
    \bar{E}_{ij}
    \,=\,
    \tbt
    { -\, \kappa_i (\bar{L}^{-2})_{ii}}
    { \kappa_j (\bar{L}^{-2})_{ij} }
    { -\, \kappa_i(\bar{L}^{-2})_{ji}}
    { \kappa_j (\bar{L}^{-2})_{jj}}
    \]
and the $ij$th entry of $\bar{L}^{-2}$ can be computed by multiplying the $i$th row of $\bar{L}^{-1}$ with the $j$th column of $\bar{L}^{-1}$. Thus, evaluation of $J_{ij}$ takes $O(n)$ operations and computation of the matrix inverse in~(\ref{eq.rank2}) requires $O(n^2)$ operations.

Since the total number of swaps for large-scale networks can be large, we follow~\cite{josboy09} and limit the maximum number of swaps with a linear function of the number of nodes $n$. Furthermore, the particular structure of networks can be exploited to reduce the required number of swaps. To illustrate this, let us consider the problem of selecting one leader in a network with $9$ nodes shown in Fig.~\ref{fig.grid9}. Suppose that all nodes in the sets $S_1 \DefinedAs \{ 1, 3, 7, 9 \}$ and $S_2 \DefinedAs \{ 2, 4, 6, 8 \}$ have the same feedback gains $\kappa_1$ and $\kappa_2$, respectively. In addition, suppose that node $5$ is chosen as a leader. Owing to symmetry, to check if selecting other nodes as a leader can improve performance we only need to swap node $5$ with one node in each set $S_1$ and $S_2$. We note that more sophisticated symmetry exploitation techniques have been discussed in~\cite{boydiaparxia09,rahjimesege09}.

    \begin{figure}
      \centering
        \includegraphics[width=0.25\textwidth]{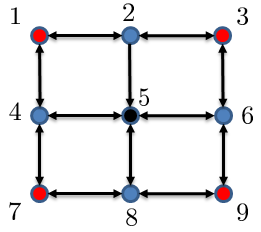}
      \caption{A lattice with $9$ nodes.}
      \label{fig.grid9}
    \end{figure}

    \begin{figure}
    \centering
    \subfloat[Lower and upper bounds resulting from convex relaxation~\eqref{CR1} and greedy algorithm, respectively.]
    {\label{fig.geometric_bounds}
    \begin{tikzpicture}
    \begin{axis}[
    	xlabel=number of leaders $N_l$]
    \addplot[color=red,mark=*]
    	table[x=leaders,y=upper] {geometric_graph_bounds.mat};
    \addplot[color=blue,mark=o]
    	table[x=leaders,y=lower] {geometric_graph_bounds.mat};
    \legend{upper bounds,lower bounds}
    \end{axis}
    \end{tikzpicture}
    }
    \subfloat[The gap between lower and upper bounds.]
    {\label{fig.geometric_gap}
    \begin{tikzpicture}
    \begin{axis}[
    	xlabel=number of leaders $N_l$]
    \addplot[color=black,mark=diamond*]
    	table[x=leaders,y=gap] {geometric_graph_bounds.mat};
    \end{axis}
    \end{tikzpicture}
    }
    \caption{Bounds on the global optimal value for noise-corrupted leader selection~\eqref{LS1} for the random network example.}
    \label{fig.random}
    \end{figure}
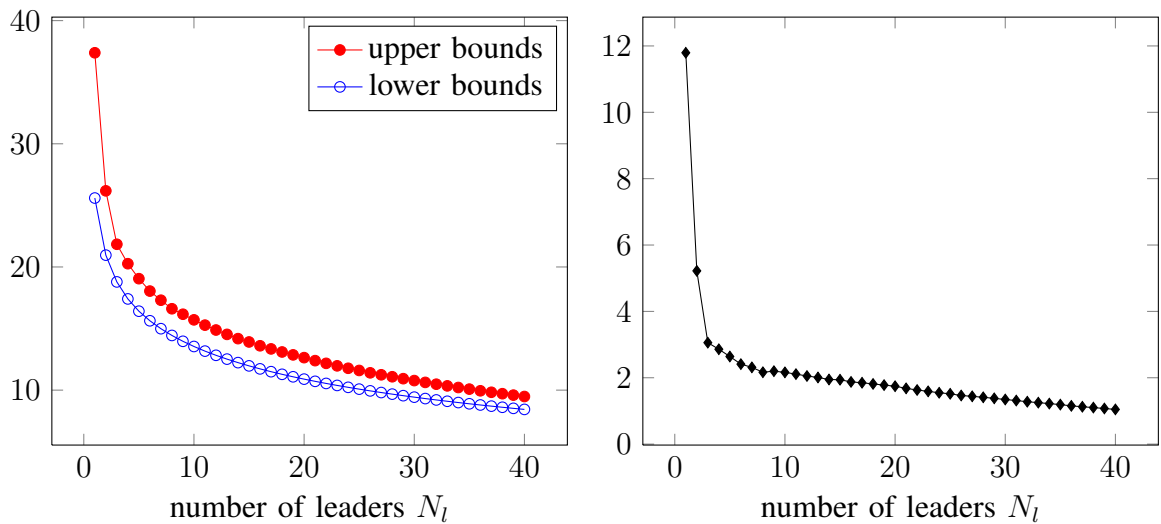

    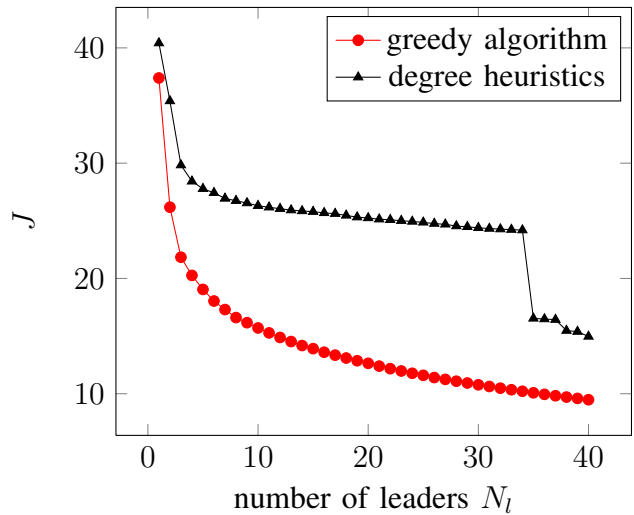
\begin{figure}
      \centering
        \begin{tikzpicture}
        \begin{axis}[
        	xlabel=number of leaders $N_l$,
        	ylabel=$J$]
        \addplot[color=red,mark=*]
        	table[x=leaders,y=greedy] {geometric_graph_greedy_degree.mat};
        \addplot[color=black,mark=triangle*]
        	table[x=leaders,y=degree] {geometric_graph_greedy_degree.mat};
        \legend{greedy algorithm, degree heuristics}
        \end{axis}
        \end{tikzpicture}
      \caption{Performance obtained using the greedy algorithm and the degree heuristics for the random network example.}
      \label{fig.random_comparison}
    \end{figure}

	\vspace*{-2ex}
\subsection{Examples}
    \label{sec.example}

We next provide two examples to illustrate performance of developed methods.

\subsubsection{A random network example}
    \label{ex.random_graph}

We consider the selection of noise-corrupted leaders in a network with $100$ randomly distributed nodes in a unit square. A pair of nodes communicate with each other if their distance is not greater than $0.2$ units. This scenario may arise in sensor networks with prescribed omnidirectional (i.e., disk shape) sensing range~\cite{mesege10,sritewluo08}.

Figure~\ref{fig.geometric_bounds} shows lower bounds resulting from the convex relaxation~\eqref{CR1} and upper bounds resulting from the greedy algorithm (i.e., the one-leader-at-a-time algorithm followed by the swap algorithm). As the number of leaders $N_l$ increases, the gap between lower and upper bounds decreases; see Fig.~\ref{fig.geometric_gap}. For $N_l = 1, \ldots, 40$, the number of swap updates ranges between $1$ and $26$ and the average number of swaps is $8$.

As shown in Fig.~\ref{fig.random_comparison}, the greedy algorithm significantly outperforms the degree-heuristics-based-selection. To gain some insight into the selection of leaders, we compare results obtained using the greedy method with the degree heuristics. As shown in Fig.~\ref{fig.random_k5_degree}, the degree heuristics chooses nodes that turn out to be in the proximity of each other. In contrast, the greedy method selects leaders that, in addition to having large degrees, are far from each other; see Fig.~\ref{fig.random_k5_greedy}. As a result, the selected leaders can influence more followers and thus more effectively improve the performance of the network.

The contrast between degree heuristics and greedy algorithms becomes even more dramatic for large number of leaders. As shown in Figs.~\ref{fig.random_k40_greedy} and~\ref{fig.random_k40_degree}, the leader sets obtained using the greedy algorithm and degree heuristics are almost {\em complements\/} of each other. While the degree heuristics clusters the leaders around the center of the network, the greedy algorithm distributes the leaders around the boundary of the network.

    \begin{figure}
      \centering
        \subfloat[Greedy algorithm: $N_l = 5$, $J = 19.0$]
        {\label{fig.random_k5_greedy}
        \includegraphics[width=0.45\textwidth]{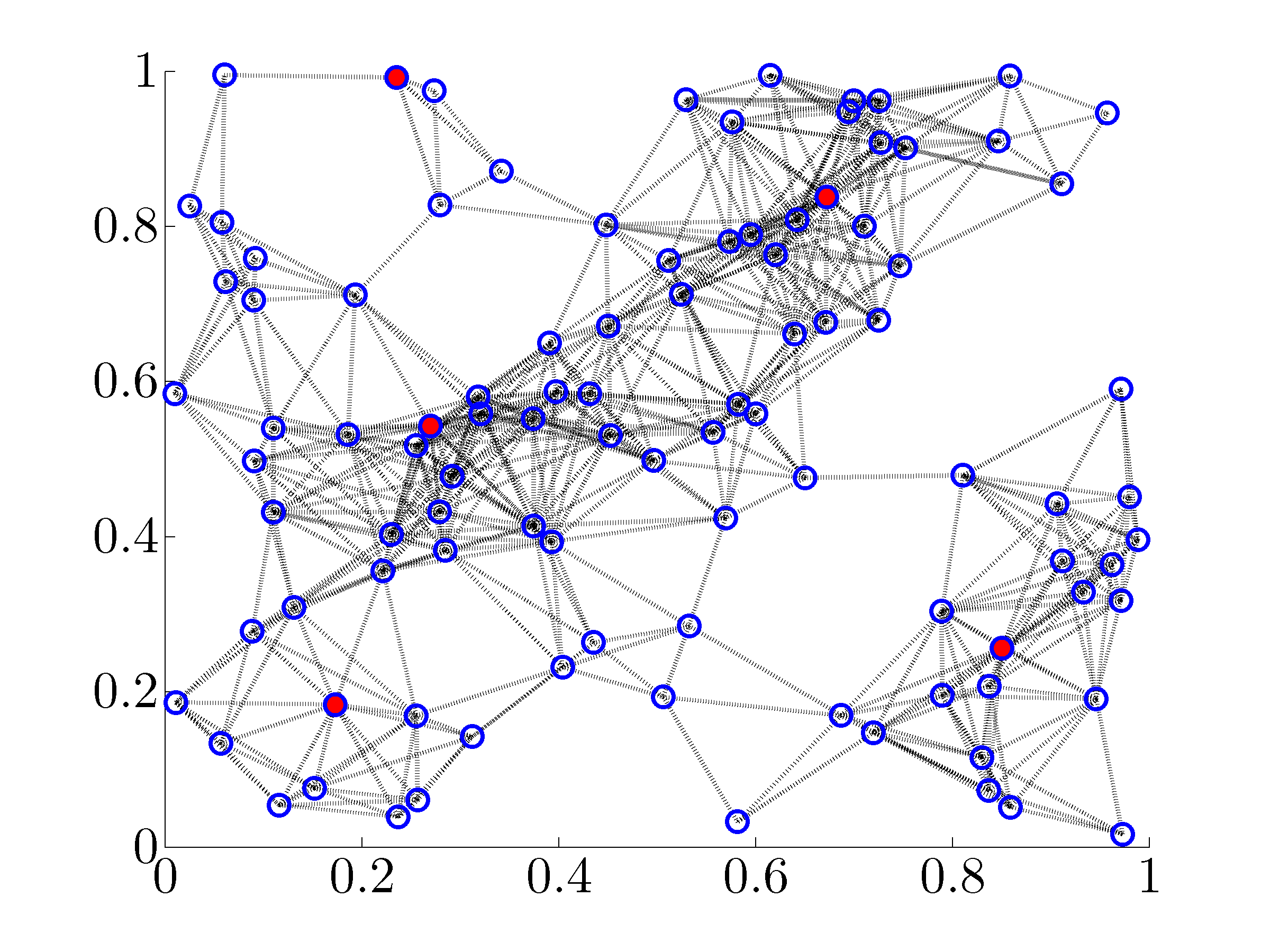}}
        \subfloat[Degree heuristics: $N_l = 5$, $J = 27.8$]
        {\label{fig.random_k5_degree}
        \includegraphics[width=0.45\textwidth]{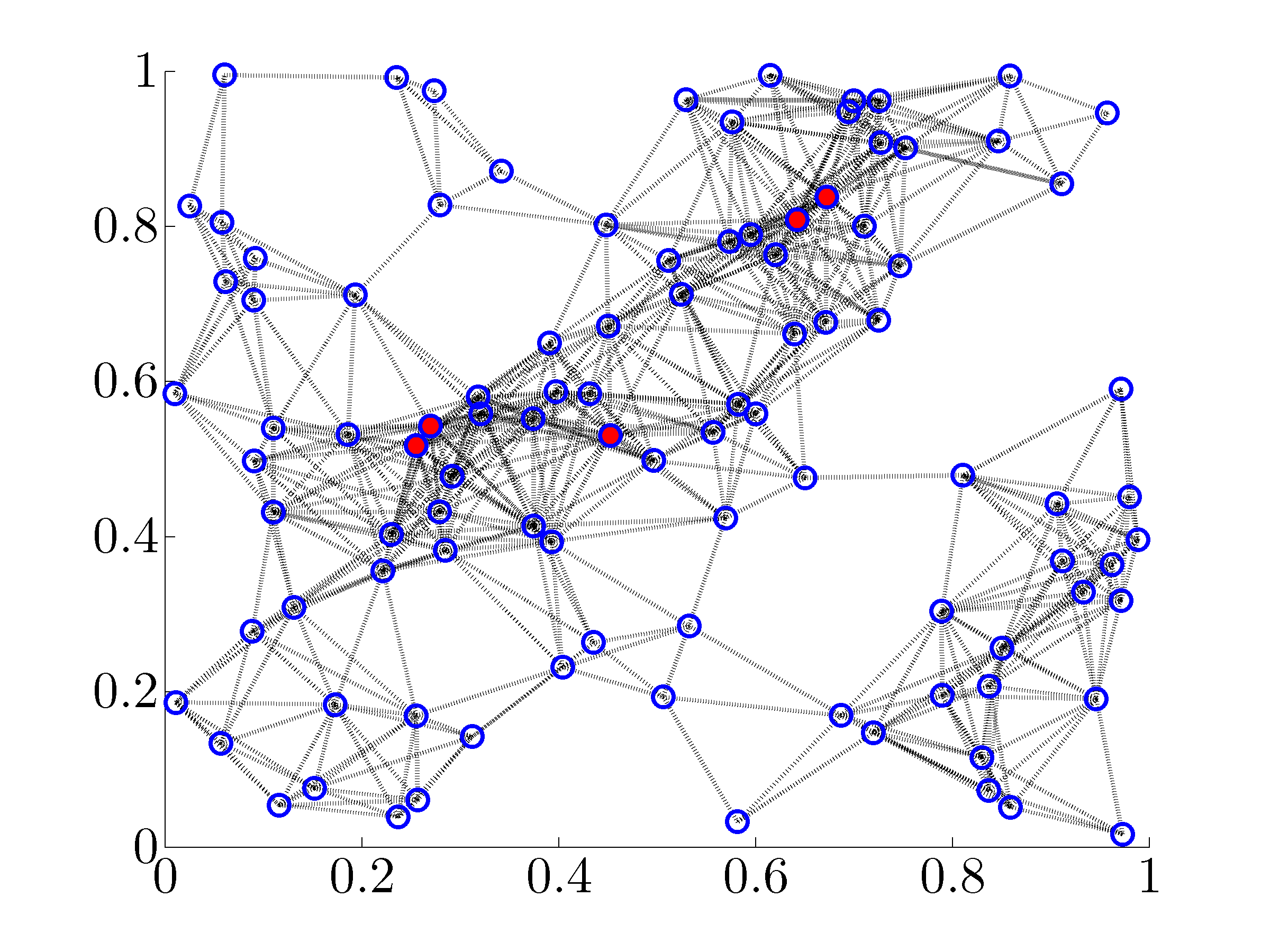}}
        \\
        \subfloat[Greedy algorithm: $N_l = 40$, $J = 9.5$]
        {\label{fig.random_k40_greedy}
        \includegraphics[width=0.45\textwidth]{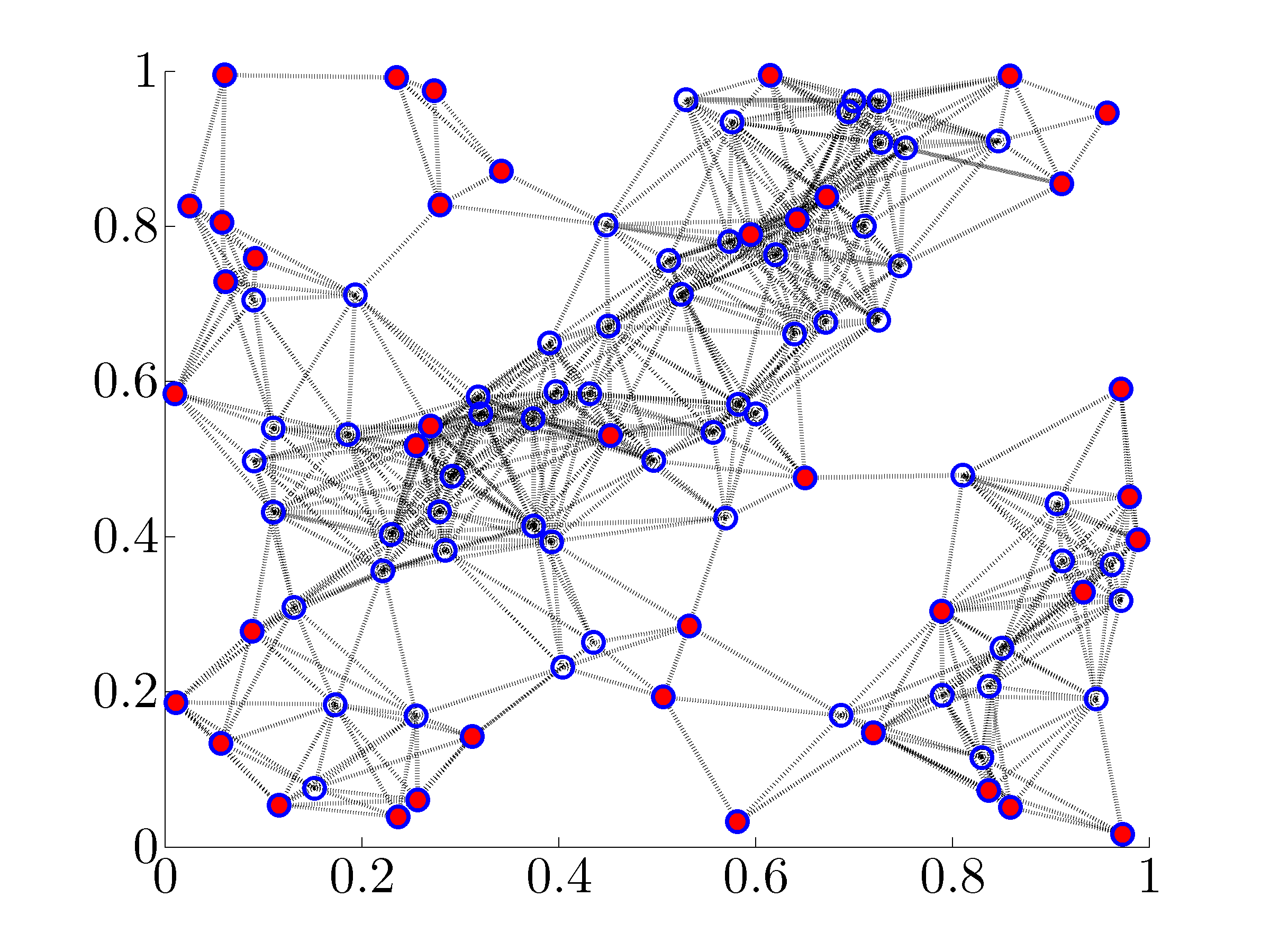}}
        \subfloat[Degree heuristics: $N_l = 40$, $J = 15.0$]
        {\label{fig.random_k40_degree}
        \includegraphics[width=0.45\textwidth]{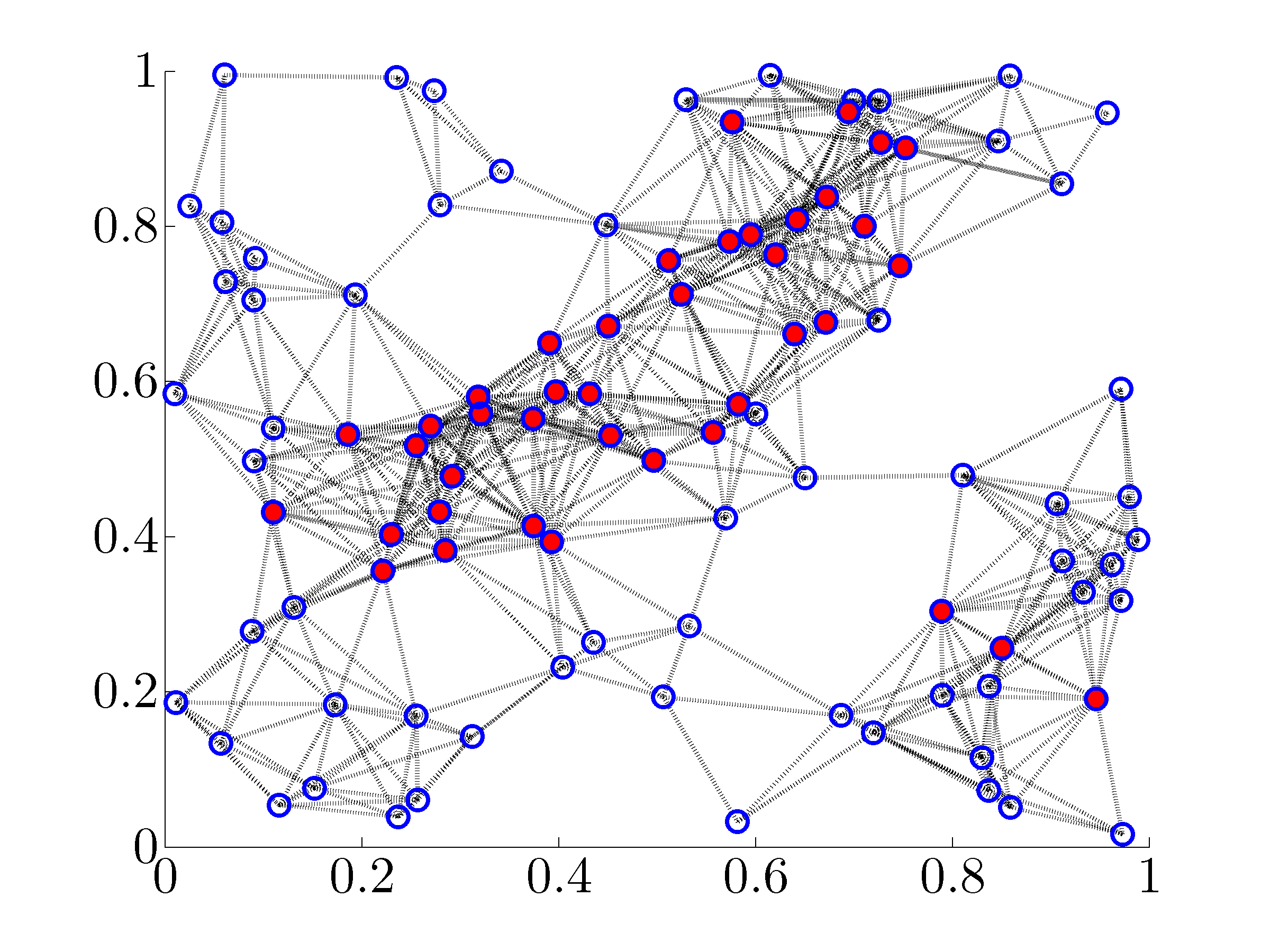}}
      \caption{Selection of leaders~(\tc{red}{$\bullet$}) for the random network example using greedy algorithm in (a) and (c) and using degree heuristics in (b) and (d).}
      \label{fig.random_leaders}
    \end{figure}

\subsubsection{A 2D lattice}

We next consider the noise-corrupted leader selection problem~\eqref{LS1} for a 2D regular lattice with $81$ nodes. Figure~\ref{fig.lattice_bounds} shows lower bounds resulting from convex relaxation~\eqref{CR1} and upper bounds resulting from the greedy algorithm. As in the random network example, the performance gap decreases with $N_l$; see Fig.~\ref{fig.lattice_gap}. For $N_l = 1, \ldots, 40$, the number of swap updates ranges between $1$ and $19$ and the average number of swaps is $10$.

Figure~\ref{fig.lattice_leaders} shows selection of leaders resulting from the greedy algorithm for different choices of $N_l$. For $N_l = 1$, the  center node $(5,5)$ provides the optimal selection of a single leader. As $N_l$ increases, nodes away from the center node are selected; for example, for $N_l = 2$, nodes $\{(3,3)$, $(7,7)\}$ are selected and for $N_l = 3$, nodes $\{(2,6)$, $(6,2)$, $(8,8)\}$ are selected. Selection of nodes farther away from the center becomes more significant for $N_l = 4$ and $N_l = 8$.

As shown in Fig.~\ref{fig.lattice_leaders}, the selection of leaders exhibits symmetry with respect to the center of the lattice. In particular, when $N_l$ is large, almost uniform spacing between the leaders is observed; see Fig.~\ref{fig.lattice_k31} for $N_l = 31$. This is in contrast to the random network example where boundary nodes were selected as leaders; see Fig.~\ref{fig.random_k40_greedy}.

    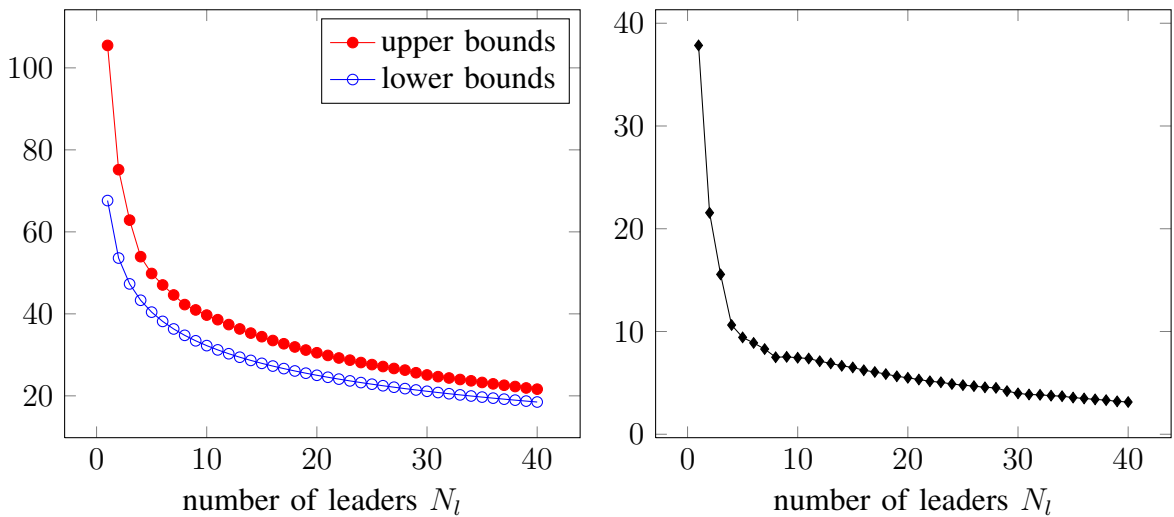
\begin{figure}
    \centering
    \subfloat[Lower and upper bounds resulting from convex relaxation~\eqref{CR1} and greedy algorithm, respectively.]
    {\label{fig.lattice_bounds}
    \begin{tikzpicture}
    \begin{axis}[
    	xlabel=number of leaders $N_l$,
    	ylabel=]
    \addplot[color=red,mark=*]
    	table[x=leaders,y=upper] {lattice_bounds.mat};
    \addplot[color=blue,mark=o]
    	table[x=leaders,y=lower] {lattice_bounds.mat};
    \legend{upper bounds,lower bounds}
    \end{axis}
    \end{tikzpicture}
    }
    \subfloat[The gap between lower and upper bounds.]
    {\label{fig.lattice_gap}
    \begin{tikzpicture}
    \begin{axis}[
    	xlabel=number of leaders $N_l$,
    	ylabel=]
    \addplot[color=black,mark=diamond*]
    	table[x=leaders,y=gap] {lattice_bounds.mat};
    \end{axis}
    \end{tikzpicture}
    }
      \caption{Bounds on the global optimal value for noise-corrupted leader selection~\eqref{LS1} for a 2D lattice.}
      \label{fig.lattice}
    \end{figure}

    \begin{figure}
      \centering
        \subfloat[$N_l = 1$, $J = 105.5$]
        {\label{fig.lattice_k1}
        \includegraphics[width=0.25\textwidth]{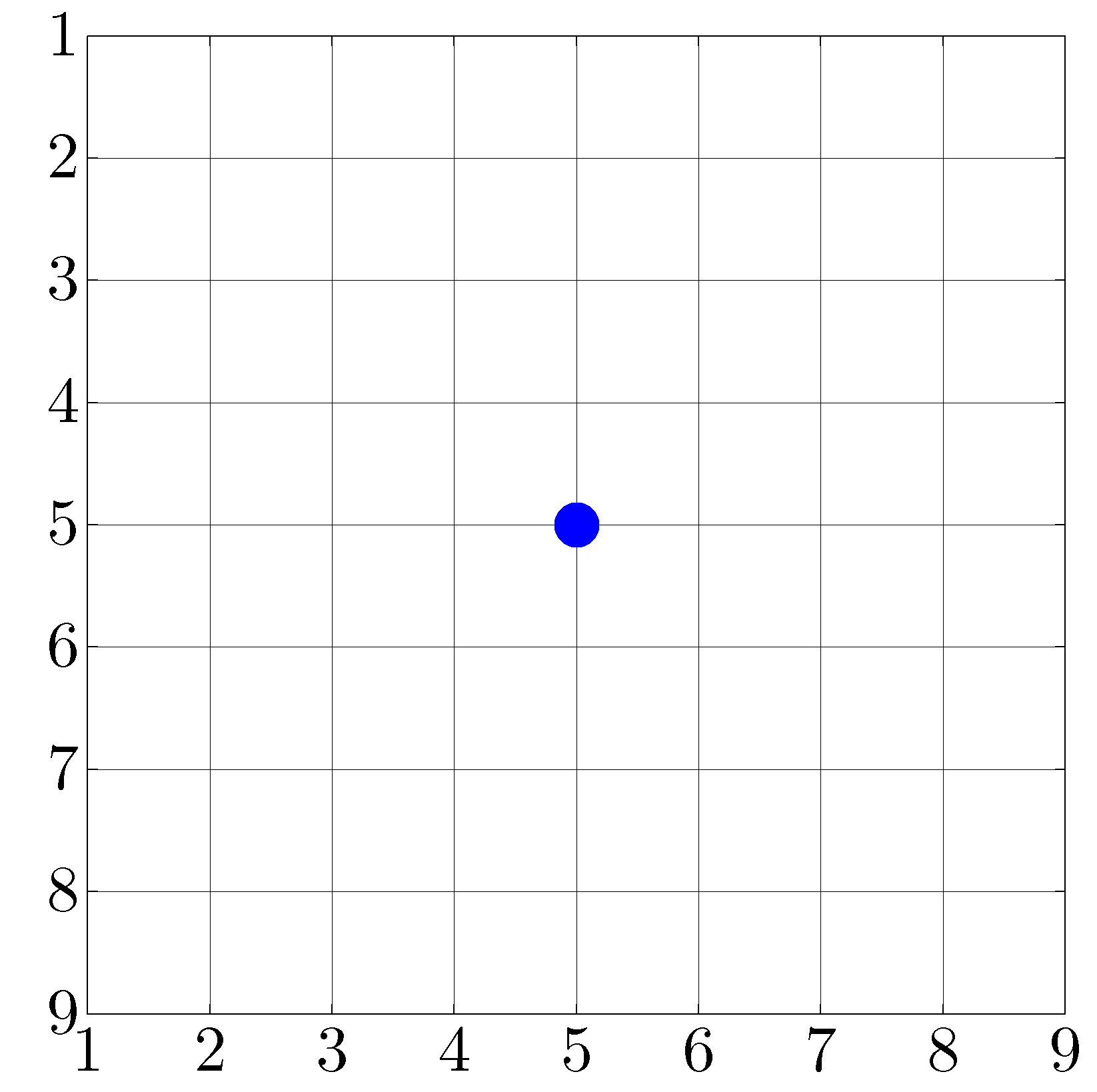}}
        \hspace{1cm}
        \subfloat[$N_l = 2$, $J = 75.2$]
        {\label{fig.lattice_k2}
        \includegraphics[width=0.25\textwidth]{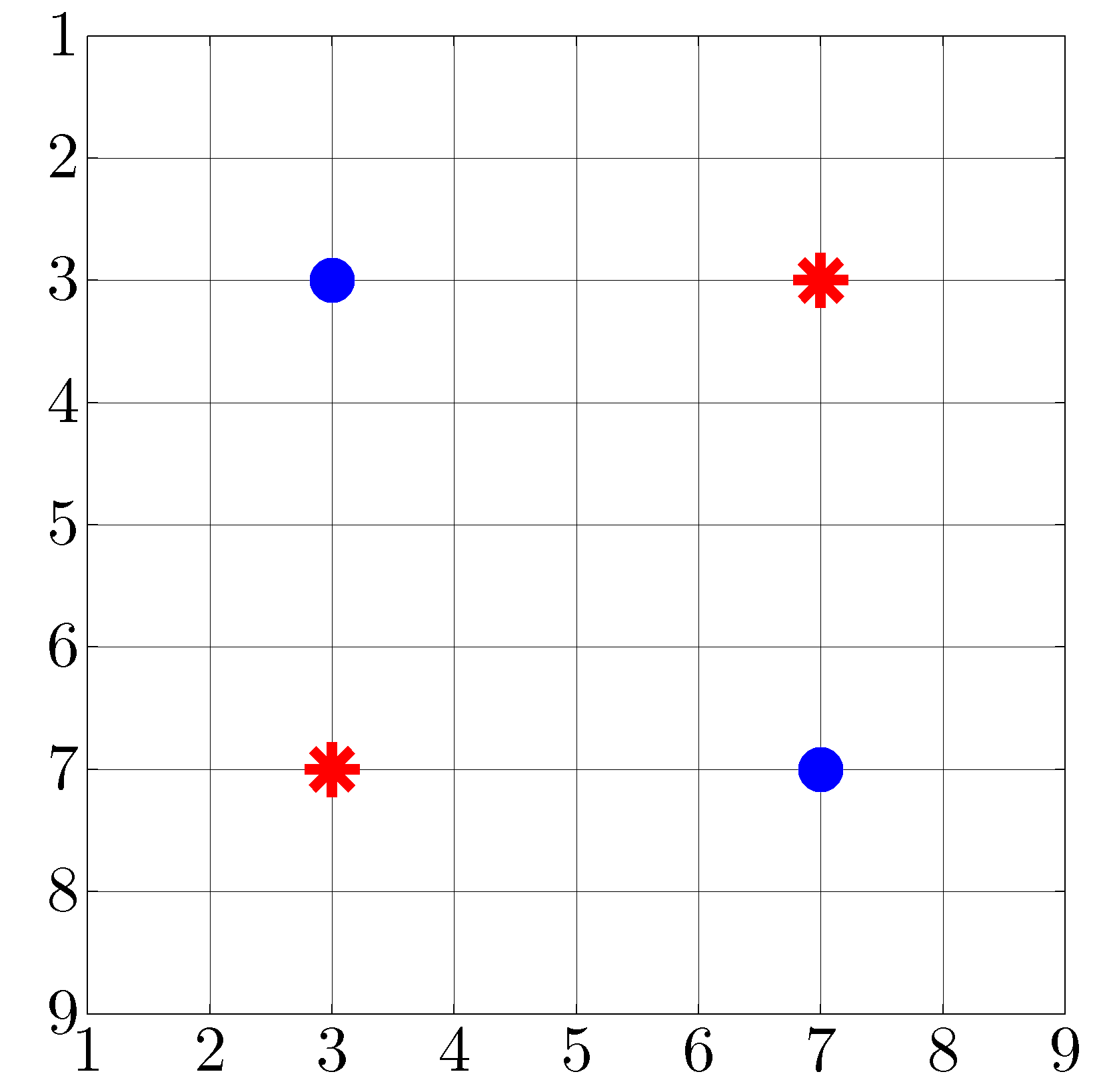}}
        \hspace{1cm}
        \subfloat[$N_l = 3$, $J = 62.9$]
        {\label{fig.lattice_k3}
        \includegraphics[width=0.25\textwidth]{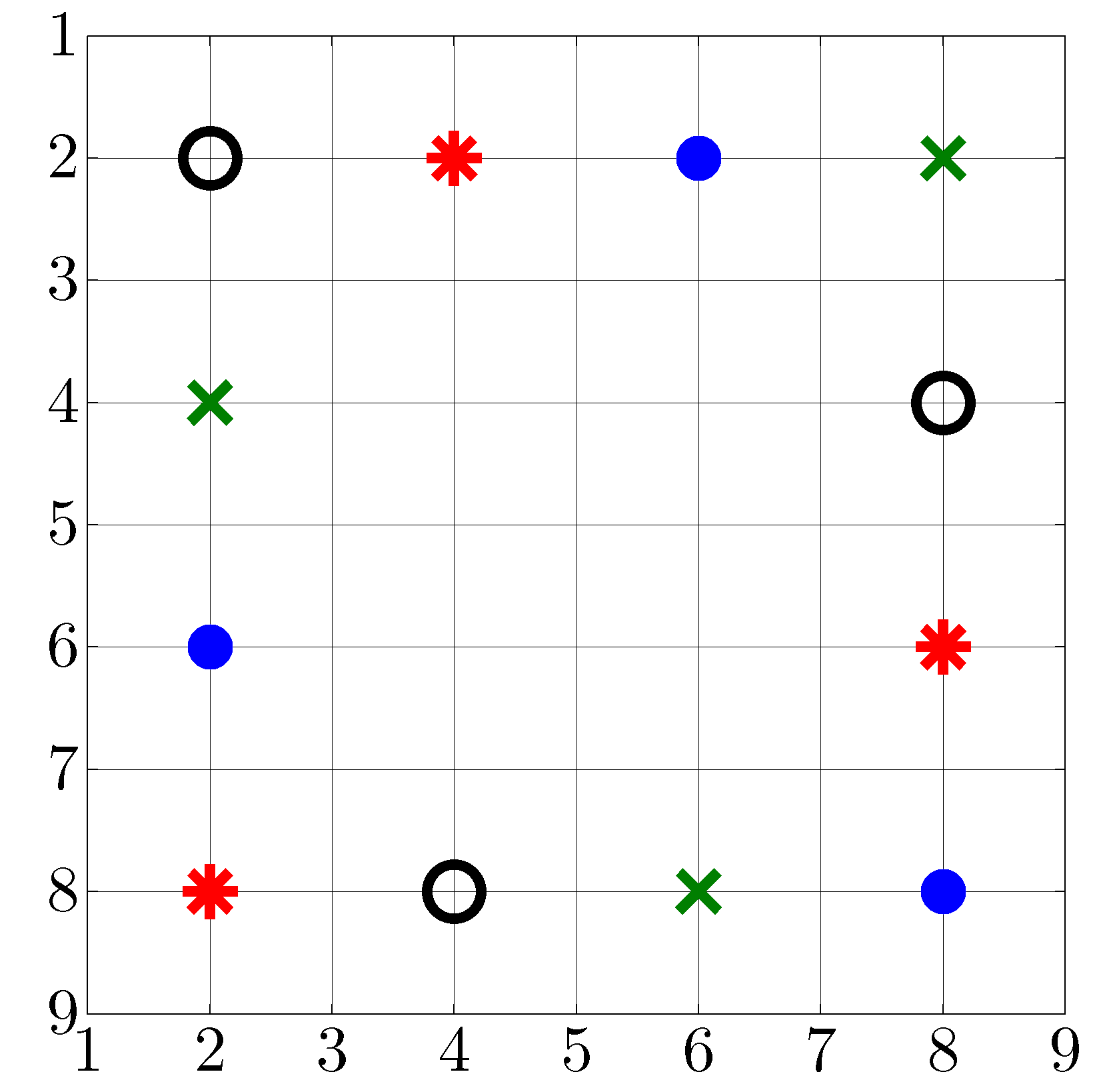}}
        \\
        \subfloat[$N_l = 4$, $J = 53.9$]
        {\label{fig.lattice_k4}
        \includegraphics[width=0.25\textwidth]{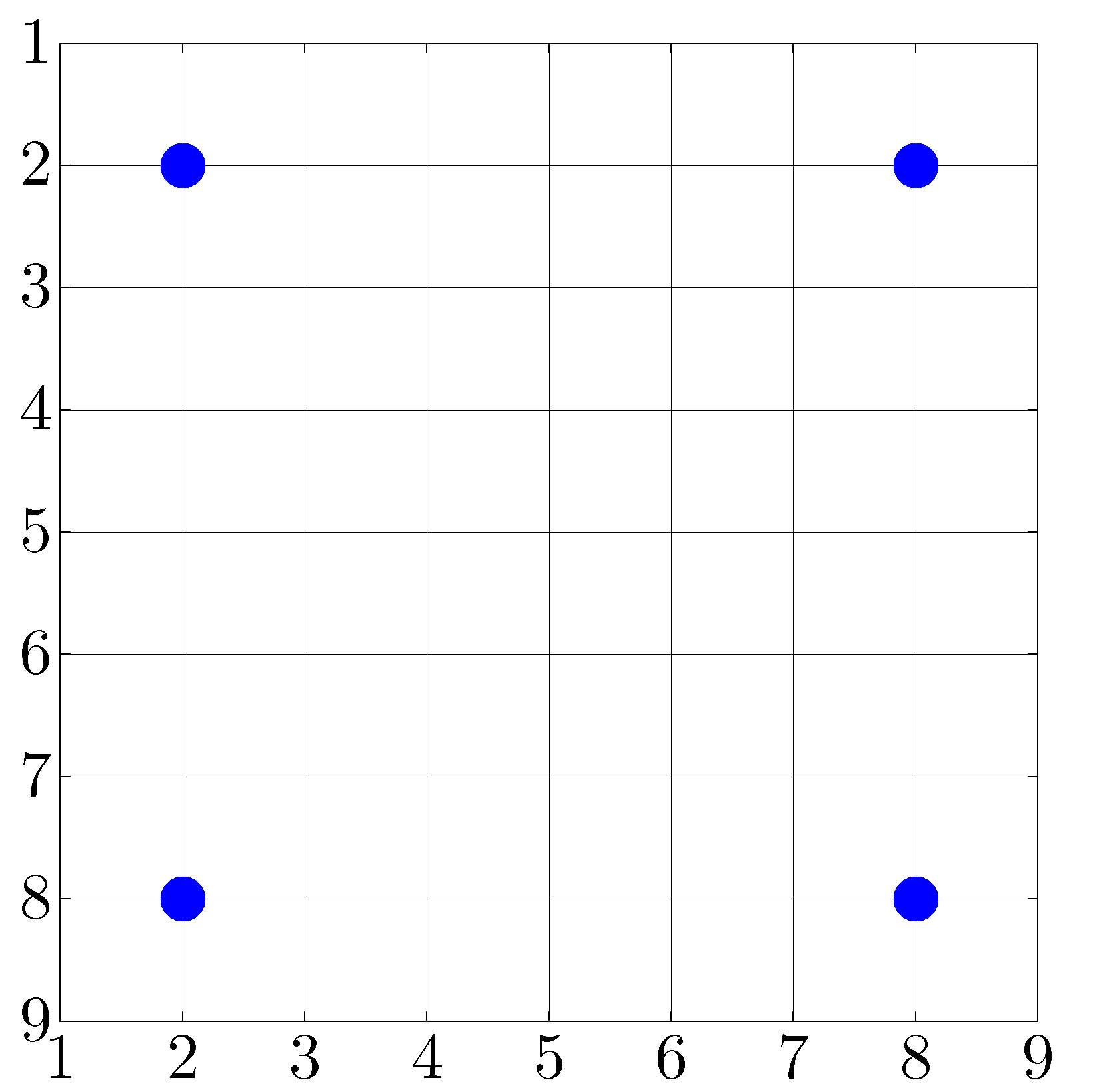}}
        \hspace{1cm}
        \subfloat[$N_l = 8$, $J = 42.3$]
        {\label{fig.lattice_k8}
        \includegraphics[width=0.25\textwidth]{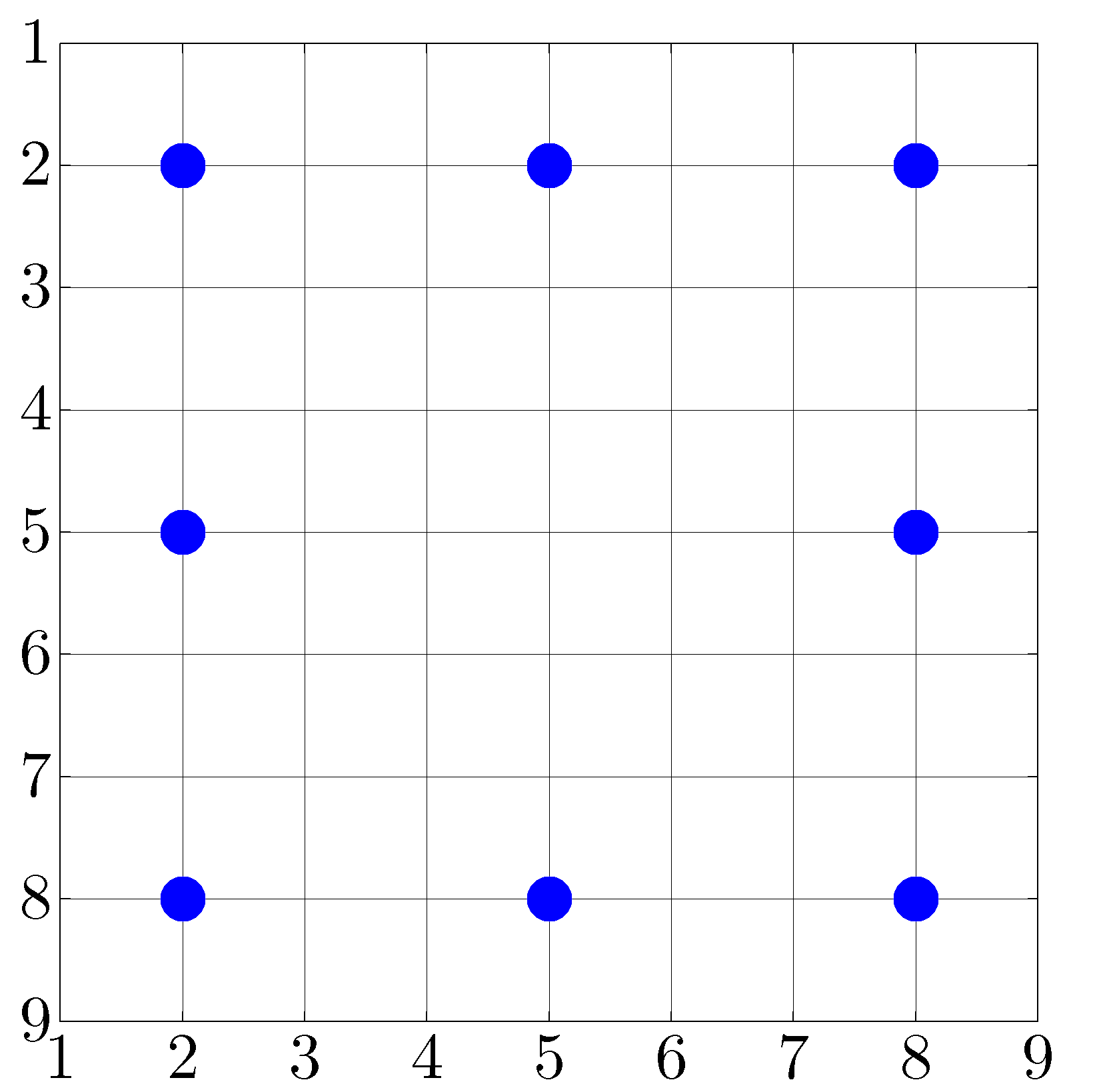}}
        \hspace{1cm}
        \subfloat[$N_l = 31$, $J = 24.7$]
        {\label{fig.lattice_k31}
        \includegraphics[width=0.25\textwidth]{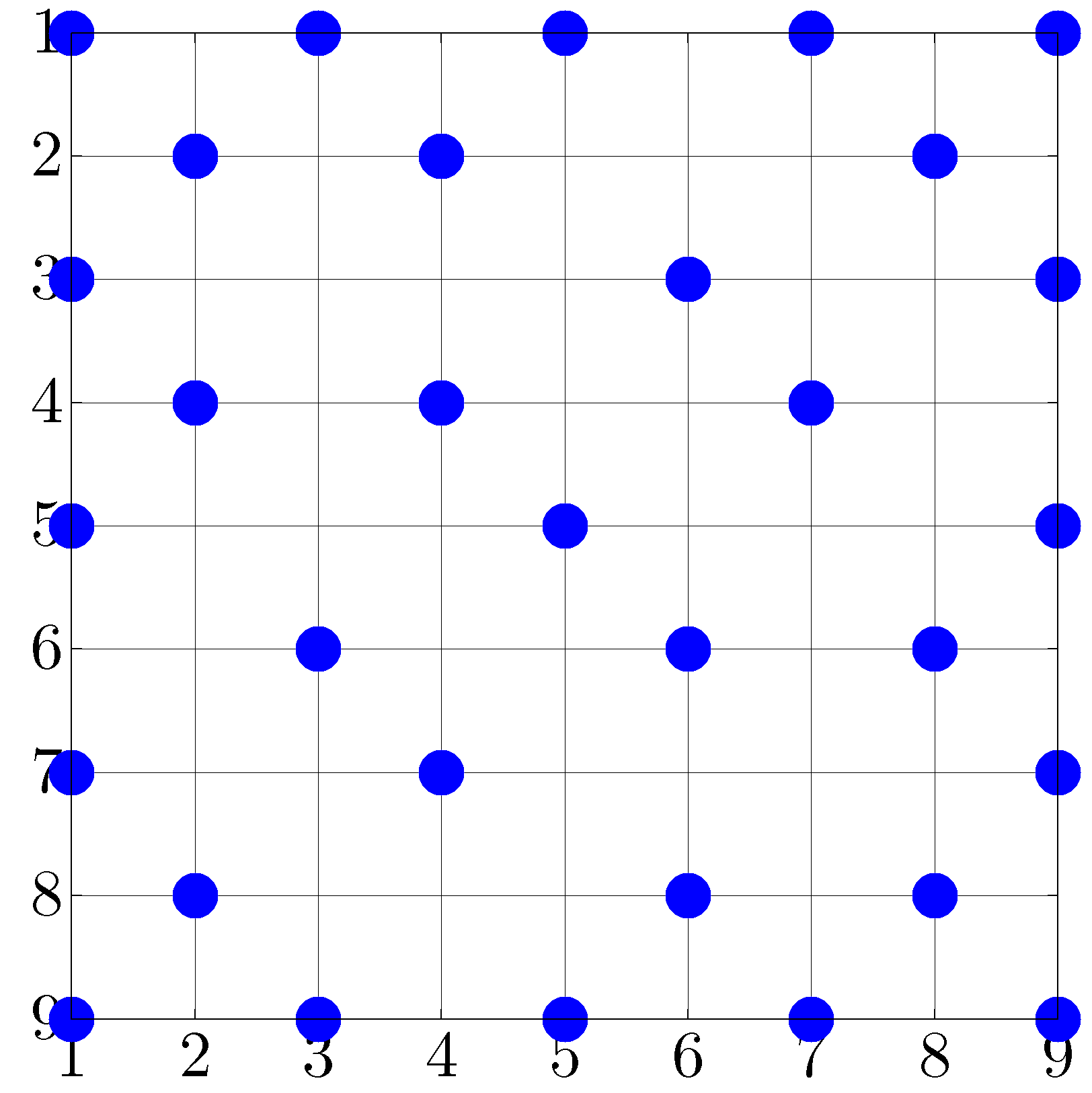}}
      \caption{Selections of noise-corrupted leaders (\tc{blue}{$\bullet$}) obtained using the one-at-a-time algorithm followed by the swap algorithm for a 2D lattice. (b) The two selections of two leaders denoted by (\tc{blue}{$\bullet$}) and (\tc{red}{$*$}) provide the same objective function $J$. (c) The four selections of three leaders denoted by (\tc{blue}{$\bullet$}), (\tc{red}{$*$}), (\tc{darkgreen}{$\times$}), and ($\circ$) provide the same $J$.}
      \label{fig.lattice_leaders}
    \end{figure}

	\vspace*{-2ex}
\section{Lower and upper bounds on global performance: Noise-free leaders}
    \label{sec.noisefree}

We now turn our attention to the noise-free leader selection problem~\eqref{LS2}. An explicit expression for the objective function $J_f$ that we develop in~\eqref{LS2} allows us to identify the source of nonconvexity and to suggest a convex relaxation. The resulting convex relaxation, which comes in the form of a semidefinite program, is used to obtain a lower bound on the global optimal value of~\eqref{LS2}. In order to increase computational efficiency, we employ the alternating direction method of multipliers to decompose the relaxed problem into a sequence of subproblems that can be solved efficiently. We also use the greedy algorithm to compute an upper bound and to identify noise-free leaders. As in the noise-corrupted leader selection problem, we take advantage of low-rank modifications to Laplacian matrices to reduce computational complexity. An example from sensor networks is provided to illustrate performance of the developed approach.

	\vspace*{-2ex}
\subsection{An explicit expression for the objective function $J_f$ in~\eqref{LS2}}
    \label{sec.explicit}

Since the objective function $J_f$ in~\eqref{LS2} is not expressed explicitly in terms of the optimization variable $x$, it is difficult to examine its basic properties (including convexity). In Proposition~\ref{pro.Jf}, we provide an alternative expression for $J_f$ that allows us to establish the lack of convexity and to suggest a convex relaxation of $J_f$.

    \begin{proposition}
    \label{pro.Jf}
For networks with at least one leader, the objective function $J_f$ in the noise-free leader selection problem~\eqref{LS2} can be written as
    \begin{align}
        J_f
        \, &= \,
        \trace
        \left(
        ( L \circ ((\dso - x)(\dso - x)^T) \,+\, \diag \left( x \right) )^{-1}
        \right)
        \, - \,
        \dso^T x
        \label{eq.Jf}
    \end{align}
where $\circ$ denotes the elementwise multiplication of matrices.
    \end{proposition}
    \begin{proof}
Let the graph Laplacian $L$ be partitioned into $2 \times 2$ block matrices which respectively correspond to the set of leaders and the set of followers
    \beq
        L \,=\, \tbt{L_l}{L_0}{L_0^T}{L_f}.
    \label{eq.Llf}
    \eeq
Furthermore, let the Boolean-valued vector $x$ be partitioned conformably
    \beq
        x \, \DefinedAs \, \obt{\dso_{N_l}^T}{0_{N_f}^T}^T
    \label{eq.x}
    \eeq
where $\dso_{N_l}$ is an $N_l$-vector with all ones, $0_{N_f}$ is an $N_f$-vector with all zeros, and
    \[
        N_f \, \DefinedAs \, n \, - \, N_l
    \]
is the number of followers. The elementwise multiplication of matrices can be used to set the rows and columns of $L$ that correspond to leaders to zero,
	 \[
        L \circ ((\dso - x)(\dso - x)^T)
        \, = \,
        \tbt{L_l}{L_0}{L_0^T}{L_f}
        \, \circ \,
        \tbt{0_{N_l \times N_l}}{0_{N_l \times N_f}}{0_{N_f \times N_l}}{\dso_{N_f \times N_f}}
        \, = \,
        \tbt{0_{N_l \times N_l}}{0_{N_l \times N_f}}{0_{N_f \times N_l}}{L_f}.
    \]
Using this expression and the definition of the vector $x$ in~\eqref{eq.x} we obtain
    \beq
    \label{eq.Lred}
        ( L \circ ((\dso - x)(\dso - x)^T) \,+\, \diag \left( x \right) )^{-1}
        \,=\,
        \tbt{I_{N_l \times N_l}}{0_{N_l \times N_f}}{0_{N_f \times N_l}}{L_f^{-1}}.
    \eeq
Finally, taking trace of~\eqref{eq.Lred} and subtracting $\dso^T x = N_l$ yields the desired result~\eqref{eq.Jf}.
    \end{proof}

Thus, the noise-free leader selection problem~\eqref{LS2} can be formulated as
    \beq
        \ba{lrcl}
        \underset{x}{\mbox{minimize}}
        &
        J_f(x)
        & = &
        \trace
        \left(
        (
        L \circ ((\dso - x)(\dso - x)^T)
        \,+\,
        \diag \left( x \right)
        )^{-1}
        \right)
        \, - \,
        N_l
        \\
        \mbox{subject to}
        &
        x_i
        & \in &
        \{0,1\},
        ~~~~~
        i \; = \; 1,\ldots,n
        \\
        &
        \dso^T x & = & N_l
        \ea
        \tag{LS2'}
        \label{LS2'}
    \eeq
where the constraint $\dso^T x = N_l$ is used to obtain the expression for the objective function $J_f$ in~\eqref{LS2'}. A counterexample can be provided to demonstrate the lack of convexity of  $J_f (x)$. In fact, it turns out that $J_f$ is not convex even if all $x_i$'s are restricted to the interval $[0,1]$. In Section~\ref{sec.convex_noisefree}, we introduce a change of variables to show that the lack of convexity of $J_f$ can be equivalently recast as a rank constraint.


	\vspace*{-2ex}
\subsection{Reformulation and convex relaxation of~\eqref{LS2'}}
    \label{sec.convex_noisefree}

By introducing a new variable $y \DefinedAs \dso - x$, we can rewrite~\eqref{LS2'} as
    \[
        \ba{lrcl}
        \underset{Y, \; y}{\mbox{minimize}}
        &
        J_f(Y,y)
        & = &
        \trace
        \left(
        ( L \circ Y \,+\, \diag \left( \dso - y \right) )^{-1}
        \right)
        \,-\, N_l
        \\
        \text{subject to}
        &
        Y & = &
        y y^T
        \\
        &
        y_i
        & \in &
        \{0,1\},
        ~~~~~
        i \; = \; 1,\ldots,n
        \\
        &
        \dso^T y & = & N_f.
        \ea
    \]
Since $Y \DefinedAs y y^T$, it follows that $Y$ is a Boolean-valued matrix with $\dso^T Y \dso = N_f^2$. Expressing these implicit constraints as
    \[
        Y_{ij}
        \, \in \,
        \{0,1\},
        ~~~
        i,j \; = \; 1,\ldots,n,
        ~~~
        \dso^T Y \dso \,=\, N_f^2
    \]
leads to the following equivalent formulation
    \[
        \ba{lrcl}
        \underset{Y, \; y}{\mbox{minimize}}
        &
        J_f(Y,y)
        & = &
        \trace
        \left(
        ( L \circ Y \,+\, \diag \left( \dso - y \right) )^{-1}
        \right)
        \,-\, N_l
        \\
        \text{subject to}
        &
        Y & = &
        y y^T
        \\
        &
        y_i
        & \in &
        \{0,1\},
        ~~~~~
        i \; = \; 1,\ldots,n
        \\
        &
        Y_{ij}
        & \in &
        \{0,1\},
        ~~~~~
        i,j \; = \; 1,\ldots,n
        \\
        &
        \dso^T y & = & N_f
        \\
        &
        \dso^T Y \dso & = & N_f^2.
        \ea
    \]
Furthermore, since
    \[
        Y \,=\, y y^T
        ~~\iff~~
        \{ \,
        Y \, \succeq \, 0,
        ~
        {\bf rank} \, (Y) \, = \, 1
        \, \}
    \]
it follows that~\eqref{LS2'} can be expressed as
    \[
        \ba{lrcl}
        \underset{Y, \; y}{\mbox{minimize}}
        &
        J_f(Y,y)
        & = &
        \trace
        \left(
        ( L \circ Y \,+\, \diag \left( \dso - y \right) )^{-1}
        \right)
        \,-\, N_l
        \\
        \text{subject to}
        &
        y_i
        & \in &
        \{0,1\},
        ~~~~~
        i \; = \; 1,\ldots,n
        \\
        &
        Y_{ij}
        & \in &
        \{0,1\},
        ~~~~~
        i,j \; = \; 1,\ldots,n
        \\
        &
        \dso^T y & = & N_f
        \\
        &
        \dso^T Y \dso & = & N_f^2
        \\
        &
        Y & \succeq &
        0,
        ~~~
        {\bf rank} \, ( Y ) \, = \, 1.
        \ea
    \]
By dropping the nonconvex rank constraint and by enlarging the Boolean set $\{0,1\}$ to its convex hull $[0,1]$, we obtain the following convex relaxation of the leader selection problem~\eqref{LS2}
    \beq
    \label{CR2}
    \tag{CR2}
        \ba{lrcl}
        \underset{Y, \; y}{\mbox{minimize}}
        &
        J_f(Y,y)
        & = &
        \trace
        \left(
        ( L \circ Y \,+\, \diag \left( \dso - y \right) )^{-1}
        \right)
        \,-\, N_l
        \\
        \text{subject to}
        &
        y_i
        & \in &
        [0,1],
        ~~~~~
        i \; = \; 1,\ldots,n
        \\
        &
        Y_{ij}
        & \in &
        [0,1],
        ~~~~~
        i,j \; = \; 1,\ldots,n
        \\
        &
        \dso^T y & = & N_f
        \\
        &
        \dso^T Y \dso & = & N_f^2
        \\
        &
        Y & \succeq &
        0.
        \ea
    \eeq

The objective function in~\eqref{CR2} is convex because it is a composition of a convex function $\trace \, (W^{-1})$ of a positive definite matrix $W$ with an affine function $W \DefinedAs L \circ Y + \diag \, (\dso - y)$ of $Y$ and $y$. The constraint set for $y$ is convex because it is the simplex set defined as
    \beq
    \label{eq.simplex_y}
    \tag{C1}
        {\cal C}_1
        \, \DefinedAs \,
        \left\{
        \,
        y
        \left| \right.
        y_i
        \, \in \,
        [0,1],
        ~~
        i \, = \, 1,\ldots,n,
        ~~
        \dso^T y
        \, = \,
        N_f
        \,
        \right\}.
    \eeq
The constraint set for $Y$ is also convex because it is the intersection of the simplex set
    \beq
    \label{eq.simplex_Y}
    \tag{C2}
        {\cal C}_2
        \, \DefinedAs \,
        \left\{\,
        Y
        \left| \right.
        Y_{ij}
        \, \in \,
        [0,1],
        ~~
        i,j \, = \, 1,\ldots,n,
        ~~
        \dso^T Y \dso
        \, = \,
        N_f^2
        \,
        \right\}
    \eeq
and the positive semidefinite cone
    \beq
    \label{eq.psd}
    \tag{C3}
    {\cal C}_3 \, \DefinedAs \, \{ \,Y \left| \right. Y \, \succeq \, 0 \, \}.
    \eeq

Following a similar procedure to that in Section~\ref{sec.convex}, we use Schur complement to cast~\eqref{CR2} as an SDP. Furthermore, since the constraints~\eqref{eq.simplex_y}-\eqref{eq.psd} are {\em decoupled\/} over $y$ and $Y$, we exploit this {\em separable\/} structure in Section~\ref{sec.ADMM} and develop an efficient algorithm to solve~\eqref{CR2}.

	\vspace*{-2ex}
\subsection{Solving the convex relaxation~\eqref{CR2} using ADMM}
    \label{sec.ADMM}

For small networks (e.g., $n \leq 30$), the convex relaxation~\eqref{CR2} can be solved using general-purpose SDP solvers, with computational complexity of order $n^6$. We next exploit the separable structure of the constraint set~\eqref{eq.simplex_y}-\eqref{eq.psd} and develop an alternative approach that is well-suited for large problems. In our approach, we use the alternating direction method of multipliers~(ADMM) to decompose~\eqref{CR2} into a sequence of subproblems which can be solved with computational complexity of order $n^3$.

Let $\phi_1(y)$ be the indicator function of the simplex set in~\eqref{eq.simplex_y},
    \[
        \phi_1(y)
        \, \DefinedAs \,
        \left\{
        \ba{ll}
        0, &  y \, \in \, {\cal C}_1
        \\
        \infty, &  y \, \notin \, {\cal C}_1.
        \ea
        \right.
    \]
Similarly, let $\phi_2(Y)$ and $\phi_3(Y)$ be the indicator functions of the simplex set~\eqref{eq.simplex_Y} and the positive semidefinite cone~\eqref{eq.psd}, respectively. Then the convex relaxation~\eqref{CR2} can be expressed as a sum of convex functions
    \beq
    \non
        \underset{Y, \; y}{\mbox{minimize}}
        ~~
        J_f(Y,y)
        \,+\,
        \phi_1(y)
        \,+\,
        \phi_2(Y)
        \,+\,
        \phi_3(Y).
    \eeq

We now introduce additional variables $\{Z,z\}$ and rewrite~\eqref{CR2} as
    \beq
    \label{eq.ADMM_form}
        \ba{ll}
        \underset{Y, \; y; \; Z, \; z}{\mbox{minimize}}
        &
        f(Y,y)
        \,+\,
        g(Z,z)
        \\[0.15cm]
        \text{subject to}
        &
        Z \,=\, Y,
        ~~~
        z \,=\, y
        \ea
    \eeq
where
    \[
    \ba{rrl}
        f(Y,y)
        & \!\! \DefinedAs \!\! &
        J_f(Y,y)
        \,+\,
        \phi_3(Y)
   	\\
	g(Z,z)
        & \!\! \DefinedAs \!\! &
        \phi_1(z) \,+\, \phi_2(Z).
        \ea
    \]
In~\eqref{eq.ADMM_form}, $f$ and $g$ are two independent functions over two different sets of variables $\{Y,y\}$ and $\{Z,z\}$, respectively. As we describe below, this separable feature of the objective function in~\eqref{eq.ADMM_form} in conjunction with the separability of the constraint set~\eqref{eq.simplex_y}-\eqref{eq.psd} is amenable to the application of the ADMM algorithm.

We form the augmented Lagrangian associated with~\eqref{eq.ADMM_form},
    \[
        {\cal L}_{\rho} (Y,y; Z,z; \Lambda,\lambda)
        \,=\,
        f(Y,y) \,+\, g(Z,z)
        \,+\,
        \langle \Lambda, Y - Z \rangle
        \,+\,
        \lambda^T (y - z)
        \,+\,
        \dfrac{\rho}{2} \,
        \| Y - Z \|_F^2
        \,+\,
        \dfrac{\rho}{2} \,
        \| y - z \|_2^2
    \]
where $\Lambda$ and $\lambda$ are Lagrange multipliers, $\rho$ is a positive scalar, $\langle \cdot, \cdot \rangle$ is the inner product of two matrices, $\langle M_1, M_2 \rangle \DefinedAs\trace(M_1^T M_2)$, and $\|\cdot\|_F$ is the Frobenius norm. To find the solution of~\eqref{eq.ADMM_form}, the ADMM algorithm uses a sequence of iterations
    \begin{subequations}
    \label{eq.ADMM_twoblock}
    \begin{align}
    \label{eq.Yy_update}
    (Y,y)^{k+1}
    \;&\DefinedAs\;
    \underset{Y, \; y}{\operatorname{arg \, min}}
    \;
    {\cal L}_\rho
    (Y,y; Z^k,z^k; \Lambda^k,\lambda^k)
    \\
    \label{eq.Zz_update}
    (Z,z)^{k+1}
    \;&\DefinedAs\;
    \underset{Z, \; z}{\operatorname{arg \, min}}
    \;
    {\cal L}_\rho
    (Y^{k+1},y^{k+1}; Z,z; \Lambda^k,\lambda^k)
    \\
    \Lambda^{k+1}
    \;&\DefinedAs\;
    \Lambda^{k}
    \,+\,
    \rho \,
    (Y^{k+1} \,-\, Z^{k+1})
    \\[-0.1cm]
    \lambda^{k+1}
    \;&\DefinedAs\;
    \lambda^{k}
    \,+\,
    \rho \,
    (y^{k+1} \,-\, z^{k+1})
    \end{align}
    \end{subequations}
until the primal and dual residuals are sufficiently small~\cite[Section 3.3]{boyparchupeleck11}
    \[
        \ba{rcl}
        \| Y^{k+1} - Z^{k+1} \|_F
        \,+\,
        \| y^{k+1} - z^{k+1} \|_2
        & \leq &
        \epsilon
        \\[0.1cm]
        \| Z^{k+1} - Z^{k} \|_F
        \,+\,
        \| z^{k+1} - z^{k} \|_2
        & \leq &
        \epsilon.
        \ea
    \]
The convergence of ADMM for convex problems is guaranteed under fairly mild conditions~\cite[Section 3.2]{boyparchupeleck11}. Furthermore, for a fixed value of parameter $\rho$, a linear convergence rate of ADMM has been established in~\cite{honluo13}. In practice, the convergence rate of ADMM can be improved by appropriately updating $\rho$ to balance the primal and dual residuals; see~\cite[Section 3.4.1]{boyparchupeleck11}.

In what follows, we show that the $(Y,y)$-minimization step~\eqref{eq.Yy_update} amounts to the minimization of a smooth convex function over the positive semidefinite cone ${\cal C}_3$. We use a gradient projection method to solve this problem. On the other hand, the $(Z,z)$-minimization step~\eqref{eq.Zz_update} amounts to projections on simplex sets ${\cal C}_1$ and ${\cal C}_2$, both of which can be computed efficiently.

\subsubsection{$(Y,y)$-minimization step}

Using completion of squares, we express the $(Y,y)$-minimization problem~\eqref{eq.Yy_update} as
    \beq
    \label{eq.Yy_prob}
        \ba{ll}
        \underset{Y, \; y}{\mbox{minimize}}
        &
        h(Y,y) \,=\,
        \trace \left( ( L \circ Y \,+\, \diag \left( \dso - y \right) )^{-1} \right)
        \, + \,
        \dfrac{\rho}{2}
        \,
        \| Y - U^k \|_F^2
        \, + \,
        \dfrac{\rho}{2}
        \,
        \| y - u^k \|_2^2
        \\
        \mbox{subject to}
        &
        Y \, \succeq \, 0
        \ea
    \eeq
where $U^k \DefinedAs Z^k - (1/\rho)\Lambda^k$ and $u^k \DefinedAs z^k - (1/\rho)\lambda^k$. A gradient projection method is used to minimize the smooth convex function $h$ in~\eqref{eq.Yy_prob} over the positive semidefinite cone $Y \succeq 0$. This iterative descent scheme guarantees feasibility in each iteration~\cite[Section 2.3]{ber99} by updating $Y$ as follows
    \beq
    \label{eq.Yproj}
        Y^{r+1}
        \, = \,
        Y^r
        \,+\,
        s^r \,
        (\bar{Y}^r  -  Y^r).
    \eeq
Here, the scalar $s^r$ is the stepsize of the $r$th gradient projection iteration and
    \beq
    \label{eq.proj_Y}
        \bar{Y}^r \,\DefinedAs\, [ \, Y^r \,-\, \nabla_Y h \, ]^+
    \eeq
is the projection of the matrix $Y^r - \nabla_Y h$ on the positive semidefinite cone ${\cal C}_3$. This projection can be obtained from an eigenvalue decomposition by replacing the negative eigenvalues with zero. On the other hand, since no constraints are imposed on $y$, it is updated using standard gradient descent
    \[
        y^{r+1} \, = \, y^r \, - \, s^r \, \nabla_y h
    \]
where the stepsize $s^r$ is the same as in~\eqref{eq.Yproj} and it is obtained, e.g., using the Armijo rule~\cite[Section 2.3]{ber99}. Here, we provide expressions for the gradient direction
    \beq
    \label{eq.grad_Yy}
        \ba{rcl}
        \nabla_Y h  
        & = &
        - \, ( L \circ Y \,+\, \diag \left( \dso - y \right) )^{-2}
        \, \circ \, L
        \,+\,
        \rho \, (Y \,-\, U^k)
        \\
        \nabla_y h 
        & = &
        \diag \left( ( L \circ Y \,+\, \diag \left( \dso - y \right) )^{-2} \right)
        \,+\,
        \rho \, (y \,-\, u^k)
        \ea
    \eeq
and note that the KKT conditions for~\eqref{eq.Yy_prob} are given by
    \[
        Y \, \succeq \, 0,~~~
        \nabla_Y h \, \succeq \, 0,~~~
        \langle Y, \, \nabla_Y h \rangle  \, = \, 0,~~~
        \nabla_y h \, = \, 0.
    \]
Thus, the gradient projection method terminates when $(Y^r,y^r)$ satisfies
    \[
        Y^r \, \succeq \, 0,~~~
        \nabla_Y h(Y^r) \, \succeq \, 0,~~~
        \langle Y^r, \, \nabla_Y h (Y^r) \rangle \, \leq \, \epsilon,~~~
        \| \nabla_y h (y^r) \|_2 \, \leq \, \epsilon.
    \]

Finally, we note that each iteration of the gradient projection method takes $O(n^3)$ operations. This is because the projection~\eqref{eq.proj_Y} on the positive semidefinite cone requires an eigenvalue decomposition and the gradient direction~\eqref{eq.grad_Yy} requires computation of a matrix inverse.





\subsubsection{$(Z,z)$-minimization step}

We now turn to the $(Z,z)$-minimization problem~\eqref{eq.Zz_update}, which can be expressed as
    \beq
    \label{eq.Zz_prob}
        \ba{ll}
        \underset{Z, \; z}{\mbox{minimize}}
        &
        \dfrac{\rho}{2}
        \,
        \| z \,-\, v^k \|_2^2
        ~ + ~
        \dfrac{\rho}{2}
        \,
        \| Z \,-\, V^k \|_F^2
        \\
        \mbox{subject to}
        &
        z \, \in \, {\cal C}_1,
        ~~~
        Z \, \in \, {\cal C}_2
        \ea
    \eeq
where $V^k \DefinedAs Y^{k+1} + (1/\rho)\Lambda^k$ and $v^k \DefinedAs y^{k+1} + (1/\rho)\lambda^k$. The separable structure of~\eqref{eq.Zz_prob} can be used to decompose it into two independent problems
    \begin{subequations}
    \label{eq.Zz_twoprob}
    \begin{align}
    \label{eq.z_prob}
        \underset{z \, \in \, {\cal C}_1}{\mbox{minimize}}
        &~~
        \dfrac{\rho}{2}
        \,
        \| z \,-\, v^k \|_2^2
    \\
    \label{eq.Z_prob}
        \underset{Z \, \in \, {\cal C}_2}{\mbox{minimize}}
        &~~
        \dfrac{\rho}{2}
        \,
        \| Z \,-\, V^k \|_2^2
    \end{align}
    \end{subequations}
whose solutions are determined by projections of $v^k$ and $V^k$ on \mbox{convex sets ${\cal C}_1$  and ${\cal C}_2$, respectively.}

In what follows, we focus on the projection on ${\cal C}_1$; the projection on ${\cal C}_2$ can be obtained in a similar fashion. For $N_f = 1$, ${\cal C}_1$ becomes a {\em probability simplex\/},
    \[
        {\cal C}_1
        \, = \,
        \left\{
        \,
        z
        \left| \right.
        z_i
        \, \in \,
        [0,1],
        ~~
        i \, = \, 1,\ldots,n,
        ~~
        \dso^T z
        \, = \,
        1
        \,
        \right\}
    \]
and customized algorithms for projection on probability simplex can be used; e.g., see~\cite{cheye11} and~\cite[Section 6.2.5]{parboy13}. Since for $N_f \geq 2$ these algorithms are not applicable, we view the simplex ${\cal C}_1$ as the intersection of the hyperplane $\{ \, z \,| ~ \dso^T z = N_f \,\}$ and the unit box $\{ \, z \,| ~ 0 \leq z \leq \dso \}$ and employ an ADMM-based alternating projection method in conjunction with simple analytical expressions developed in~\cite[Section 6.2]{parboy13}; see Appendix~\ref{app.quad_ADMM} for details.

	\vspace*{-2ex}
\subsection{Greedy algorithm to obtain an upper bound}

Having determined a lower bound on the global optimal value of~\eqref{LS2} by solving the convex relaxation~\eqref{CR2}, we next quantify the performance gap and provide a computationally attractive way for selecting leaders. As in the noise-corrupted case, we use the one-leader-at-a-time algorithm followed by the swap algorithm to compute an upper bound. Rank-$2$ modifications to the resulting Laplacian matrices allow us to compute the inverse of $L_f$ using $O(n^2)$ operations.

Let $[L]_{i}$ be the principal submatrix of $L$ obtained by deleting its $i$th row and column. To select the first leader, we compute
    \[
        J_1^i
        \,=\,
        \trace
        \left(
        [L]_i^{-1}
        \right),
        ~~~
        i = 1,\ldots,n
    \]
and assign the node, say $v_1$, that achieves the minimum value of $\{J_1^i\}$. After choosing $s$ noise-free leaders $\nu = \{v_1,\ldots,v_s\}$, we compute
    \[
        J_{s+1}^i
        \,=\,
        \trace
        \left(
        [L]_{\nu \, \cup \, i}^{-1}
        \right),
        ~~~
        i \notin \nu
    \]
and choose node $v_{s+1}$ that achieves the minimum value of $\{ J_{s+1}^i \}$. We repeat this procedure until all $N_l$ leaders are selected.

For $N_l \ll n$, the one-at-a-time greedy algorithm that ignores the low-rank structure requires $O(n^4 N_l)$ operations. We next exploit the low-rank structure to reduce complexity to $O(n^3 N_l)$ operations. The key observation is that the difference between two {\em consecutive\/} principal submatrices $[L]_i$ and $[L]_{i+1}$ leads to a rank-$2$ matrix. \mbox{To see this, let us partition the Laplacian matrix as}
    \[
        L
        \,=\,
        \left[
        \ba{cccc}
        L_1        &  c_i  & c_{i+1}   & L_0        \\
        c_i^T      &  a_i  & d_i       & b_i^T      \\
        c_{i+1}^T  &  d_i  & a_{i+1}   & b_{i+1}^T  \\
        L_0^T      &  b_i  & b_{i+1}   & L_2
        \ea
        \right]
        ~
        \ba{l}
        \leftarrow ~ \text{$i$th row}
        \\
        \leftarrow ~ \text{$(i+1)$th row}
        \ea
    \]
where the $i$th column of $L$ consists of $\{ c_i$, $a_i$, $d_i$, $b_i\}$ and the $(i+1)$th column consists of $\{ c_{i+1}$, $d_i$, $a_{i+1}$, $b_{i+1} \}$. Deleting the $i$th row and column and deleting the $(i+1)$th row and column respectively yields
    \beq
    \label{eq.rank2_noisefree}
        [L]_i
        \,=\,
        \thbth{L_1}{c_{i+1}}{L_0}{c_{i+1}^T}{a_{i+1}}{b_{i+1}^T}{L_0^T}{b_{i+1}}{L_2},
        ~~~
        [L]_{i+1}
        \,=\,
        \thbth{L_1}{c_i}{L_0}{c_i^T}{a_i}{b_i^T}{L_0^T}{b_i}{L_2}.
    \eeq
Thus, the difference between two consecutive principal submatrices of $L$ can be written as
    \[
        [L]_{i+1}
        \,-\,
        [L]_i
        \, = \,
        e_i \xi_i^T
        \,+\,
        \xi_i e_i^T
    \]
where $e_i$ is the $i$th unit vector and
    $
        \xi_i^T
        \DefinedAs
        [\, c_{i}^T - c_{i+1}^T ~~ \frac{1}{2}(a_{i} - a_{i+1}) ~~ b_{i}^T - b_{i+1}^T \,].
    $
Hence, once $[L]_{i}^{-1}$ is determined, computing $[L]_{i+1}^{-1}$ via matrix inversion lemma takes $O(n^2)$ operations; cf.~\eqref{eq.rank2}. The selection of the first leader requires one matrix inverse and $n-1$ times rank-$2$ updates, resulting in $O(n^3)$ operations. For $N_l \ll n$, the total cost of the greedy algorithm is thus reduced to $O(n^3 N_l)$ operations.

As in Section~\ref{sec.swap}, after selecting $N_l$ leaders using the one-leader-at-a-time algorithm we employ the swap algorithm to further improve performance. Similar to the noise-corrupted case, a swap between a noise-free leader and a follower leads to a rank-$2$ modification to the reduced Laplacian $L_f$. Thus, after a swap, the evaluation of the objective function $J_f$ can be carried out with $O(n^2)$ operations. If $L$ is partitioned as in~\eqref{eq.Llf}, a swap between leader $i$ and follower $N_l + j$ amounts to replacing (i) the $j$th row of $L_f$ with the $i$th row of $L_0$; and (ii) the $i$th column of $L_f$ with the $i$th column of $L_0^T$. Thus, a swap introduces a rank-$2$ modification to $L_f$.

	\vspace*{-2ex}
\subsection{An example}

We consider a network with $200$ randomly distributed nodes in a C-shaped region within a unit square; see Fig.~\ref{fig.Cshape_leaders}. A pair of nodes communicates with each other if their distance is not greater than $0.1$ units. This example was used in~\cite{sritewluo08} as a benchmark for testing algorithms for the sensor localization problem. Lower and upper bounds on the global optimal value of the noise-free leader selection problem~\eqref{LS2} are computed using approaches developed in this section. For $N_l = 1,\ldots,10$, the number of the swap updates ranges from $1$ to $16$ and the average number of swaps is $8$.

As shown in Fig.~\ref{fig.Cshape_bounds}, the gap between lower and upper bounds is a decreasing function of $N_l$. The greedy algorithm selects leaders that have large degrees and that are geographically far from each other; see Fig.~\ref{fig.Cshape_leaders}. Similar leader selection strategies have been observed in the noise-corrupted case of Section~\ref{sec.example}. For the C-shaped network, we note that the noise-free and noise-corrupted formulations lead to almost identical selection of leaders.

    \begin{figure}
      \centering
        \subfloat[Lower and upper bounds resulting from convex relaxation~\eqref{CR2} and greedy algorithm, respectively.]
        {\label{fig.Cshape_noisefree_bounds}
        \begin{tikzpicture}
    	\begin{axis}[xlabel=number of leaders $N_l$, mark size=3pt]
    	\addplot[color=red,mark=*] coordinates {
    		(1,337.3832)
    		(2,140.7931)
    		(3,92.5191)
    		(4,59.1369)
    		(5,53.5477)
    		(6,47.3370)
    		(7,42.2660)
    		(8,40.5564)
    		(9,38.7599)
    		(10,37.4680)
    	};
        \addlegendentry{upper bounds}
    	\addplot[color=blue,mark=o] coordinates {
       (1,43.8347)
       (2,37.9354)
       (3,34.6638)
       (4,32.1978)
       (5,30.0475)
       (6,28.0940)
       (7,26.3638)
       (8,24.5237)
       (9,22.7038)
       (10,20.9197)
    	};
        \addlegendentry{lower bounds}
    	\end{axis}
        \end{tikzpicture}}
        ~~~
        \subfloat[The gap between lower and upper bounds.]
        {\label{fig.Cshape_noise_corrupted_bounds}
                \begin{tikzpicture}
    	\begin{axis}[xlabel=number of leaders $N_l$,mark size=4pt]
    	\addplot[color=black,mark=diamond*] coordinates {
    		(1,293.5485)
    		(2,102.8577)
    		(3,57.8553)
    		(4,26.9391)
    		(5,23.5002)
    		(6,19.2430)
    		(7,15.9022)
    		(8,16.0327)
    		(9,16.0562)
    		(10,16.5483)
    	};
    	\end{axis}
        \end{tikzpicture}}
     \caption{Bounds on the global optimal value for noise-free leader selection~\eqref{LS2} in a C-shaped network.}
      \label{fig.Cshape_bounds}
\end{figure}
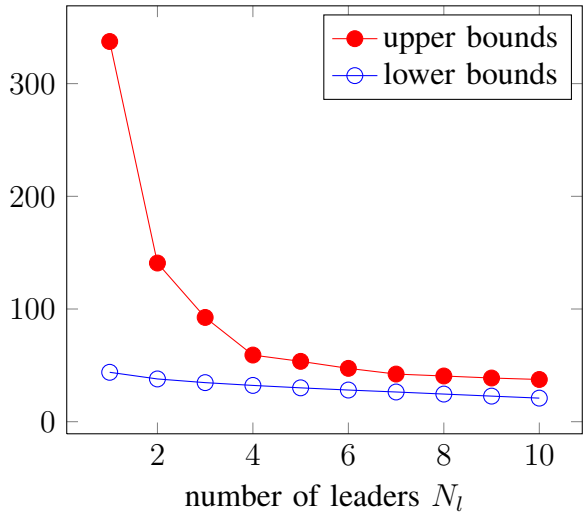
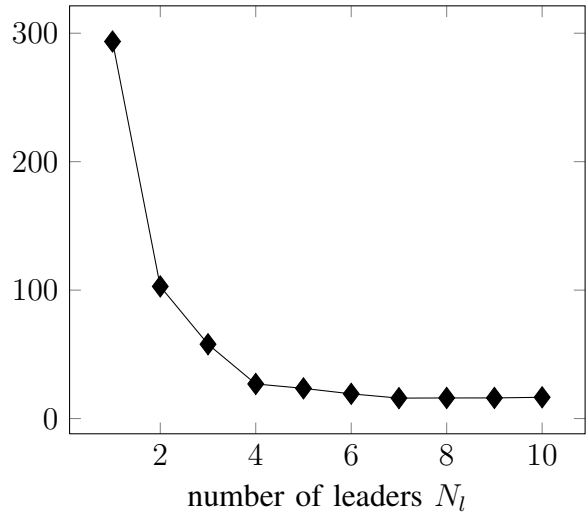

	\vspace*{-2ex}
\section{Concluding remarks}
    \label{sec.conclude}

The main contribution of this paper is the development of efficient algorithms for the selection of leaders in large stochastically forced consensus networks. For both noise-corrupted and noise-free formulations, we focus on computing lower and upper bounds on the global optimal value. Lower bounds are obtained by solving convex relaxations and upper bounds result from simple but efficient greedy algorithms.

Even though the convex relaxations can be cast as semidefinite programs and solved using general-purpose SDP solvers, we take advantage of the problem structure (such as separability of constraint sets) and develop customized algorithms for large-scale networks. We also improve the computational efficiency of greedy algorithms by exploiting the properties of low-rank modifications to Laplacian matrices. Several examples ranging from regular lattices to random networks are provided to illustrate the effectiveness of the developed algorithms. 

We are currently applying the developed tools for leader selection in different types of networks, including small-world and social networks~\cite{farzhalinjovCDC12,farlinzhajovACC13}. Furthermore, the flexibility of our framework makes it well-suited for quantifying performance bounds and selecting leaders in problem formulations with alternative objective functions~\cite{claalobuspoo12,clabuspooCDC12}. An open question of theoretical interest is whether leaders can be selected based on the solutions of the convex relaxations~\eqref{CR1} and~\eqref{CR2}. Since our computations suggest that the solution $Y^* \succeq 0$ to~\eqref{CR2} has a small number of dominant eigenvalues, it of interest to quantify the level of conservatism of the lower bounds that result from these low-rank solutions and to investigate scenarios under which~\eqref{CR2} yields a rank-1 solution. The use of randomized algorithms~\cite{kisluoluo09,luomasoyezha10} may provide a viable approach to addressing the former question.

	 \begin{figure}
      \centering
        \subfloat[$N_l = 3$, $J_f = 92.5$]
        {\label{fig.Cshape_noisefree_3leaders}
        \includegraphics[width=0.5\textwidth]{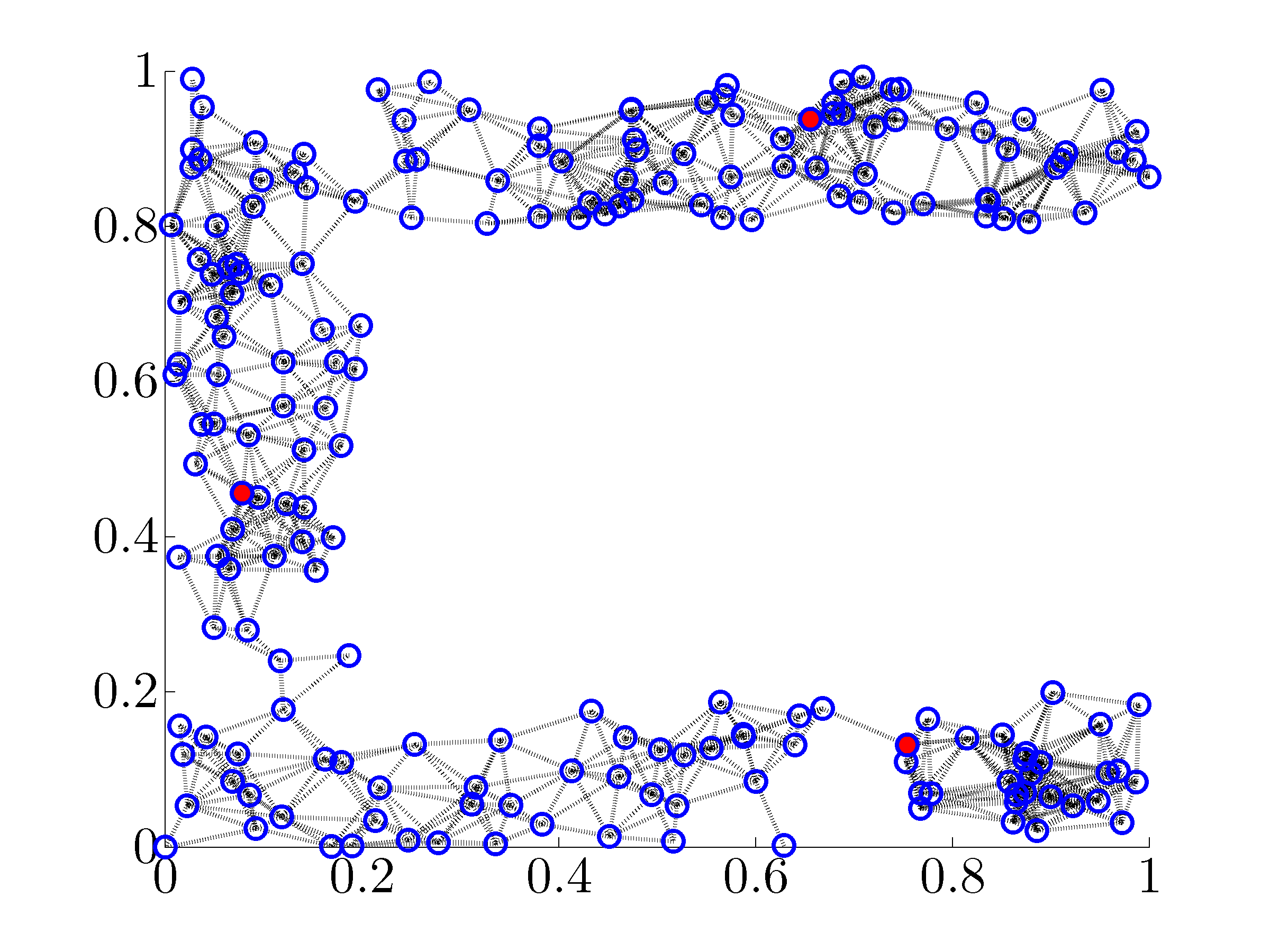}}
        \subfloat[$N_l = 9$, $J_f = 41.2$]
        {\label{fig.Cshape_noise_corrupted_9leaders}
        \includegraphics[width=0.5\textwidth]{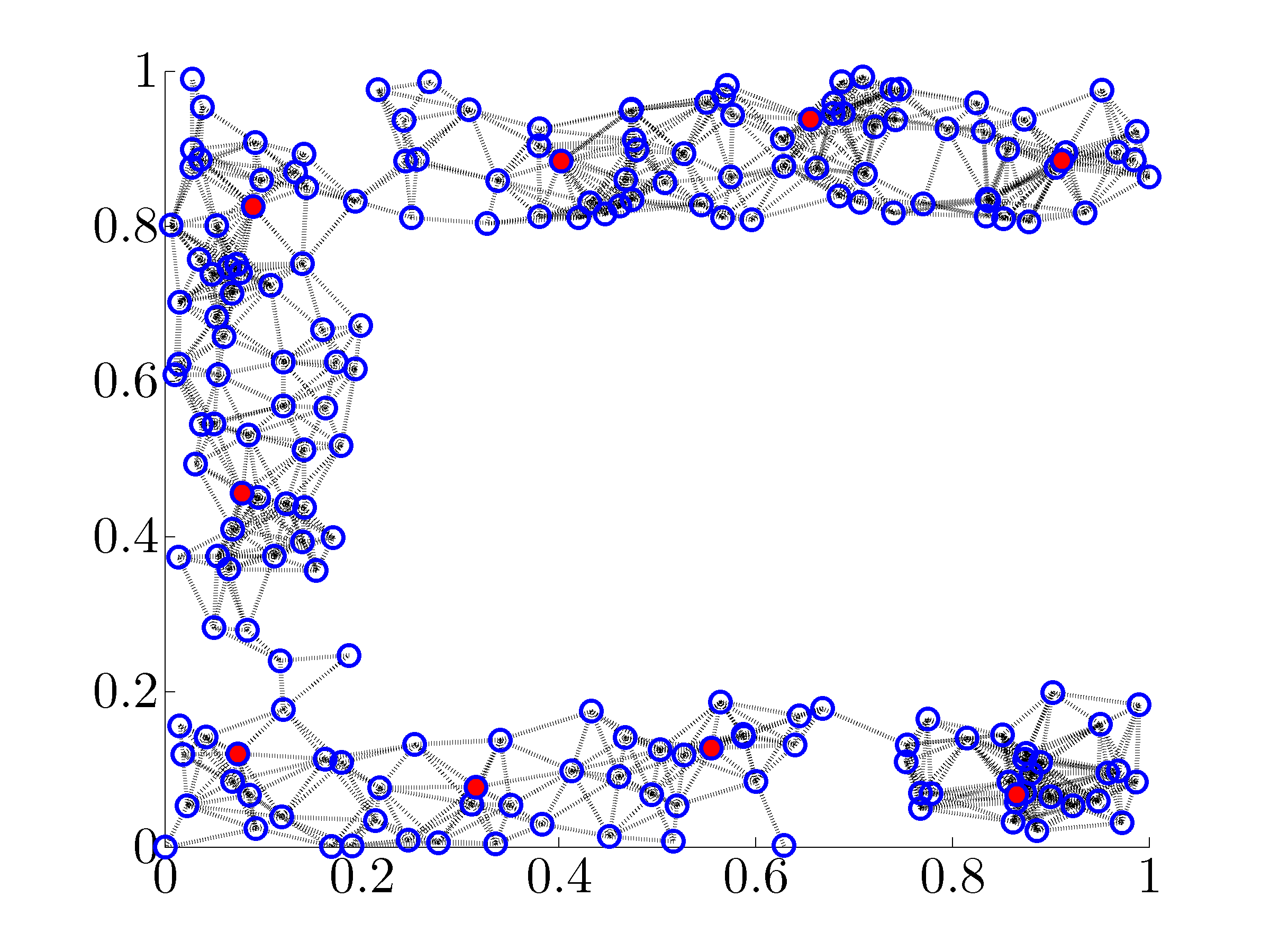}}
      \caption{Selection of noise-free leaders (\tc{red}{$\bullet$}) using the greedy algorithm for the C-shaped network.}
      \label{fig.Cshape_leaders}
    \end{figure}

	\vspace*{-2ex}
\appendix

	\vspace*{-1ex}	
\subsection{Connection between noise-free and noise-corrupted formulations}
    \label{app.nf-vs-nc}
    	
Partitioning $\psi$ into the state of the leader nodes $\psi_l$ and the state of the follower nodes $\psi_f$ brings system~(\ref{eq.state}) to the following form\footnote{Since the partition is performed with respect to the indices of the $0$ and $1$ elements of $x$, the matrix $D_x$ does not show in~(\ref{eq.partition}).}
    \beq
    \label{eq.partition}
        \tbo{\dot{\psi}_l}{\dot{\psi}_f}
        \, = \, - \,
        \tbt{L_l + D_{\kappa_l}}{L_0}{L_0^T}{L_f}
        \tbo{\psi_l}{\psi_f}
        \, + \,
        \tbo{w_l}{w_f}.
    \eeq
Here, $D_{\kappa_l} \DefinedAs \diag \, (\kappa_l)$ and $\kappa_l \in \mathbb{R}^{N_l}$ is the vector of feedback gains associated with the leaders. Taking the trace of the inverse of the $2 \times 2$ block matrix in~(\ref{eq.partition}) yields
    \[
    J
    \; = \;
    \trace
    \left(
    L_f^{-1}
    \,+\,
    L_f^{-1} \, L_0^T \, S_{\kappa_l}^{-1} \, L_0 \, L_f^{-1} \,+\, S_{\kappa_l}^{-1}
    \right)
    \]
where
    \[
        S_{\kappa_l}
        \; \DefinedAs \;
        L_l
        \,+\,
        D_{\kappa_l}
        \,-\,
        L_0 \, L_f^{-1} \, L_0^T
    \]
is the Schur complement of $L_f$. Since $S_{\kappa_l}^{-1}$ vanishes as each component of the vector $\kappa_l$ goes to infinity, the variance of the network in this case is determined by the variance of the followers,
    $
        J_{f}
        =
        \trace
        \left(
        L_f^{-1}
        \right).
   $
Here, $L_f$ denotes the reduced Laplacian matrix obtained by removing all rows and columns  that correspond to the leaders from $L$.	

	\vspace*{-2ex}
\subsection{Equivalence between leader selection and sensor selection problems}
    \label{app.sensor}

We next show that the problem of choosing $N_l$ absolute position measurements among $n$ sensors to minimize the variance of the estimation error in Section~\ref{sec.sensor-selection} is equivalent to the noise-corrupted leader selection problem~\eqref{LS1}.

Given the measurement vector $y$ in~(\ref{eq.y}), the linear minimum variance unbiased estimate of $\psi$ is determined by~\cite[Chapter 4.4]{lue68}
    \[
        \hat{\psi} \,=\, (E_r W_r^{-1} E_r^T + E_a (E_a^T W_a E_a)^{-1} E_a^T)^{-1} ( E_r W_r^{-1} y_r + E_a (E_a^T W_a E_a)^{-1} y_a)
    \]
with the covariance of the estimation error
    \[
        \Sigma
        \,=\,
        {\cal E} ( (\psi - \hat{\psi})(\psi - \hat{\psi})^T )
        \,=\,
        (E_r W_r^{-1} E_r^T + E_a (E_a^T W_a E_a)^{-1} E_a^T)^{-1}.
    \]
Furthermore, let us assume that $W_r = I$ and $W_a = D_\kappa^{-1}$. The choice of $W_a$ indicates that a larger value of $\kappa_i$ corresponds to a more accurate absolute measurement of sensor $i$. Then
    \[
        ( E_a^T W_a E_a )^{-1}
        \, = \,
        (E_a^T D_\kappa^{-1} E_a)^{-1}
        \, = \,
        E_a^T D_\kappa E_a
    \]
and thus,
    \[
        \Sigma
        \, = \,
        (E_r E_r^T + E_a E_a^T D_\kappa E_a E_a^T)^{-1}.
    \]
Since $E_a E_a^T$ is a diagonal matrix with its $i$th diagonal element being $1$ for $i \in {\cal I}_a$ and $E_r E_r^T$ is the Laplacian matrix of the relative measurement graph, it follows that
    \[
        D_x \,=\, E_a E_a^T,
        ~~~
        L \,=\, E_r E_r^T,
        ~~~
        \Sigma
        \, = \,
        (L \,+\, D_x D_\kappa D_x)^{-1}
        \, = \,
        (L \,+\, D_\kappa D_x)^{-1}
    \]
where $D_x D_\kappa D_x = D_\kappa D_x$ because $D_x$ and $D_\kappa$ commute and $D_x D_x = D_x$. Therefore, we have established the equivalence between
the noise-corrupted leader selection problem~\eqref{LS1} and the problem of choosing $N_l$ sensors with absolute position measurements such that the variance of the estimation error is minimized.

To formulate an estimation problem that is equivalent to the noise-free leader selection problem~\eqref{LS2}, we follow~\cite{barhes07} and assume that the positions of $N_l$ sensors are known {\em a priori\/}. Let $\psi_l$ denote the positions of these {\em reference sensors\/} and let $\psi_f$ denote the positions of the other sensors. We can thus write the relative measurement equation~(\ref{eq.relative}) as
    \[
        y_r \,=\, E_r^T \psi \,+ \, w_r \,=\, E_l^T \psi_l \,+\, E_f^T \psi_f \,+\, w_r
    \]
and the linear minimum variance unbiased estimate of $\psi_f$ is given by
    \[
        \hat{\psi}_f
        \,=\,
        (E_f E_f^T)^{-1} E_f W_r^{-1}
        \,
        (y_r \,-\, E_l^T \psi_l)
    \]
with covariance of the estimation error
    $
        \Sigma_{f} = (E_f E_f^T)^{-1}.
    $
Identifying $E_f E_f^T$ with $L_f$ in the Laplacian matrix
    \[
        L \, = \, E_r E_r^T
        \, = \,
        \tbt{E_l E_l^T}{E_l E_f^T}{E_f E_l^T}{E_f E_f^T}
        \,=\,
        \tbt{L_l}{L_0}{L_0^T}{L_f}
    \]
establishes equivalence between problem~\eqref{LS2} and the problem of assigning $N_l$ sensors with known reference positions to minimize the variance of the estimation error of sensor network.

	\vspace*{-2ex}
\subsection{Customized interior point method for~\eqref{CR1}}
    \label{app.IPM}

We begin by augmenting the objective function in~\eqref{CR1} with log-barrier functions associated with the inequality constraints on $x_i$
    \beq
    \label{eq.approx}
    \ba{ll}
    \underset{x}{\mbox{minimize}}
    &
    q (x)
    \, = \,
    \tau \,
    \trace
    \,\big(
    (L \,+\, D_\kappa D_x )^{-1}
    \big)
    \, + \,
    \ds{\sum_{i \, = \, 1}^n}
    \big(
    - \,
    \log(x_i)
    \, - \,
    \log(1 - x_i)
    \big)
    \\
    \text{subject to}
    &
    \dso^T x \,=\, N_l.
    \ea
    \eeq
As the positive scalar $\tau$ increases to infinity, the solution of the approximate problem~(\ref{eq.approx}) converges to the solution of the convex relaxation~\eqref{CR1}~\cite[Section 11.2]{boyvan04}. We solve a sequence of problems~(\ref{eq.approx}) by gradually increasing $\tau$, and by starting each minimization using the solution from the previous value of $\tau$. We use Newton's method to solve~(\ref{eq.approx}) for a fixed $\tau$, and the Newton direction is given by
    \beq
    \ba{rrl}
    x_{\rm nt}
    & \!\! = \!\! &
    -\,
    (\nabla^2 q)^{-1} \nabla q
    \,-\,
    \delta (\nabla^2 q)^{-1} \dso
    \\[0.15cm]
    \delta
   & \!\! \DefinedAs \!\! &
    -
    \dfrac
    {\dso^T (\nabla^2 q)^{-1} \nabla q}
    {\dso^T (\nabla^2 q)^{-1} \dso}.
    \ea
    \non
    \eeq
Here, the expressions for the $i$th entry of the gradient direction $\nabla q$ and for the Hessian matrix are given by
    \begin{align*}
    (\nabla q)_i
    \; =& \;
    - \, \tau \,
    \kappa_i \, ( (L + D_\kappa D_x)^{-2} )_{ii}
    \,-\,
    x_i^{-1}
    \,-\,
    (x_i - 1)^{-1}
    \\
    \nabla^2 q
    \; =& \;
    2 \tau \,
    ( D_\kappa (L + D_\kappa D_x)^{-2} D_\kappa ) \circ  (L + D_\kappa D_x)^{-1}
    \, + \,
    \diag \left( x_i^{-2} \,+\, (1-x_i)^{-2} \right).
    \end{align*}

We next examine complexity of computing the Newton direction $x_{\rm nt}$. The major cost of computing $\nabla^2 q$ is to form $(L + D_\kappa D_x)^{-2}$, which takes $(7/3)n^3$ operations to form $(L + D_\kappa D_x)^{-1}$ and $n^3$ operations to form $(L + D_\kappa D_x)^{-2}$. Computing $x_{\rm nt}$ requires solving two linear equations,
    \[
    (\nabla^2 q)
    \,
    y
    \,=\,
    -\nabla q,
    ~~~
    (\nabla^2 q)
    \,
    z
    \,=\,
    -\dso
    \]
which takes $(1/3)n^3$ operations using Cholesky factorization. Thus, the computation of each Newton step requires $(7/3 + 1 + 1/3)n^3 = (11/3)n^3$ operations.


	\vspace*{-2ex}
\subsection{Solving~\eqref{eq.z_prob} using ADMM}
    \label{app.quad_ADMM}

Since the solution of~\eqref{eq.z_prob} does not depend on the value of $\rho$, and since the constraint set is the intersection of the hyperplane and the unit box, we can express~\eqref{eq.z_prob} as
    \beq
    \label{eq.zw_prob}
        \ba{ll}
        \underset{z, \; w}{\mbox{minimize}}
        &
        \dfrac{1}{2} \,
        \| z \,-\, v^k \|_2^2
        \,+\, \phi_4(z)
        \,+\, \phi_5(w)
        \\
        \mbox{subject to}
        &
        z \,-\, w \,=\, 0.
        \ea
    \eeq
Here, $\phi_4$ and $\phi_5$ are the indicator functions of the hyperplane $\{ z \left| \right. \dso^T z = N_f \}$ and the box $\{ w \left| \right. 0 \leq w \leq \dso \}$, respectively. The augmented Lagrangian associated with~\eqref{eq.zw_prob} is given by
    \[
        {\cal L}_\varrho (z,w,\lambda)
        \, = \,
        \dfrac{1}{2} \,
        \| z \,-\, v^k \|_2^2
        \,+\, \phi_4(z)
        \,+\, \phi_5(w)
        \,+\, \langle \lambda, \, z \,-\, w \rangle
        \,+\, \dfrac{\varrho}{2} \, \| z \,-\, w \|_2^2
    \]
and the ADMM algorithm uses the sequence of iterations
    \begin{subequations}
    \label{eq.quad_ADMM}
    \begin{align}
    \label{eq.z_update}
    z^{s+1}
    \,&\DefinedAs\,
    \underset{z}{\operatorname{arg \, min}}
    \;
    \left(
    \frac{1}{2} \, \| z \,-\, v^k \|_2^2
    \, + \, \phi_4(z)
    \, + \,
    \frac{\varrho}{2} \, 
    \|z \,-\, (w^s \,-\, \lambda^s/\varrho)\|_2^2
    \right)
    \\
    \label{eq.w_update}
    w^{s+1}
    \,&\DefinedAs\,
    \underset{w}{\operatorname{arg \, min}}
    \;
    \left(
    \phi_5(w)
    \, + \,
    \frac{\varrho}{2} \, 
    \| w \,-\, (z^{s+1} \,+\, \lambda^s/\varrho) \|_2^2
    \right)
    \\
    \lambda^{s+1}
    \,&\DefinedAs\,
    \lambda^{s}
    \,+\,
    \varrho \,
    (z^{s+1} \,-\, w^{s+1})
    \end{align}
    \end{subequations}
until $\| z^{s+1} - w^{s+1} \|_2 \leq \epsilon$ and $\| w^{s+1} - w^s \|_2 \leq \epsilon$. By solving the KKT conditions for~\eqref{eq.z_update}, we obtain an analytical solution
    \beq
    \label{eq.z_sol}
        z^{s+1}
        \,=\,
        \left( \varrho \, w^s \,-\, \lambda^s \,+\, v^k \,-\, \eta \, \dso \right)
        /
        \left( \varrho \,+\, 1 \right)
    \eeq
where the scalar $\eta$ is given by
    \[
        \eta \,=\, \left( \dso^T (\varrho \, w^s \,-\, \lambda^s \,+\, v^k) \,-\, (\varrho \,+\, 1) N_f \right)/n.
    \]
On the other hand, the solution to~\eqref{eq.w_update} is determined by the projection of $\mu = z^{s+1} + \lambda^s / \varrho$ on the box $\{ w \,|\; 0 \leq w \leq \dso \}$,
    \beq
    \label{eq.w_sol}
        w_i^{s+1}
        \,=\,
        \left\{
        \ba{ll}
        1,      &    \mu_i > 1 \\
        \mu_i,  &    0 \leq \mu_i \leq 1 \\
        0,      &    \mu_i < 0.
        \ea
        \right.
    \eeq
Both the solution~\eqref{eq.z_sol} and the projection~\eqref{eq.w_sol} take $O(n)$ operations.

\end{document}